\documentclass[a4paper,12pt]{article}

\usepackage{amsfonts,amsthm,amssymb,amsmath}
\usepackage[colorlinks=true,linkcolor=blue,citecolor=blue]{hyperref}
\usepackage{fourier}

\newtheorem{theorem}{Theorem}[section]
\newtheorem{proposition}[theorem]{Proposition}

\newtheorem{lemma}[theorem]{Lemma}
\theoremstyle{definition}

\newtheorem{remark}[theorem]{Remark}

\DeclareMathOperator{\mon}{mon}
\DeclareMathOperator{\card}{card}
\DeclareMathOperator{\re}{Re}
\DeclareMathOperator{\spa}{span}

\title{Monomial expansions of $H_{p}$--functions in infinitely many variables}

\author{Andreas Defant\footnote{Institut f\"ur Mathematik. Universit\"at Oldenburg. D-26111 Oldenburg (Germany) } 
\and Leonhard Frerick\footnote{Fachbereich IV - Mathematik, Universit\"{a}t Trier, D-54294 Trier (Germany)} 
\and Manuel Maestre\footnote{Dep. An\'alisis Matem\'atico. Fac. Matem\'aticas. Universidad de Valencia. 46100 Burjassot (Spain)} \and 
Pablo Sevilla-Peris \footnotemark[1] \ \footnote{IUMPA. Universitat Polit\`ecnica de Val\`encia. 46010 Valencia (Spain)}}
\date{}

\begin{document}
\maketitle

\footnotetext{The  first, third and fourth  authors were  supported  by  MICINN and FEDER Project MTM2008-03211. The third author was
also supported by Prometeo 2008/101 and by MEC Grant  PR2010-0374.}

\begin{abstract}
\noindent Each bounded holomorphic function on the infinite dimensional polydisk $\mathbb{D}^\infty$, $f \in H_\infty( \mathbb{D}^\infty)$,  defines  a formal
monomial series expansion that in general does not converge to $f$.
The set   $\mon H_\infty( \mathbb{D}^\infty)$ contains all
$ z $'s in which the  monomial series expansion of each function $f \in H_\infty(\mathbb{D}^\infty)$ sums up to $f(z)$. Bohr, Bohnenblust and Hille,
showed that it  contains
$\ell_{2} \cap \mathbb{D}^\infty$, but does not contain any of the slices $\ell_{2+\varepsilon} \cap \mathbb{D}^\infty$. This was done
in the context of Dirichlet series and our
article is very much inspired by  recent deep developments in this direction.
Our main contribution shows that $z \in \mon H_\infty( \mathbb{D}^\infty)$ whenever
$\overline{\lim} \big( \frac{1}{\log n} \sum_{j=1}^{n} z^{* 2}_{j} \big)^{1/2} < 1/\sqrt{2}$, and conversely
$\overline{\lim} \big( \frac{1}{\log n} \sum_{j=1}^{n} z^{* 2}_{j} \big)^{1/2} \leq 1$ for each $z \in \mon H_\infty( \mathbb{D}^\infty)$.
The Banach space $H_\infty(\mathbb{D}^\infty)$ can be identified with the Hardy space $H_\infty(\mathbb{T}^\infty)$;
this  motivates a study of sets of monomial convergence of  $H_p$-functions on
$\mathbb{T}^\infty$ (consisting of all $z$'s in $\mathbb{D}^{\infty}$ for which the series $\sum \hat{f}(\alpha) z^{\alpha}$ converges). We show that
$\mon H_\infty( \mathbb{T}^\infty) = \mon H_\infty( \mathbb{D}^\infty)$ and $\mon H_{p}( \mathbb{T}^\infty) = \ell_{2} \cap \mathbb{D}^\infty$
for $1 \leq p < \infty$ and give a representation of $H_{p}( \mathbb{T}^\infty)$ in terms of holomorphic functions on $\mathbb{D}^{\infty}$. This
links our circle of ideas with well-known results due to  Cole and Gamelin.

\end{abstract}

\section{Introduction}

Hilbert in \cite{Hi09} was among the very first who started a systematic study  of the concept of analyticity
for functions in infinitely many
variables. According to Hilbert, an analytic function in infinitely many variables is
a $\mathbb{C}$-valued function defined on the infinite dimensional polydisc $\mathbb{D}^{\infty}$
which  has a pointwise convergent monomial series expansion:
\begin{equation}\label{monexp}
f(z)=\sum_{\alpha\in \mathbb{N}_{0}^{(\mathbb{N})}} c_{\alpha} z^{\alpha}\,\,, \,\,\,\,z \in \mathbb{D}^{\infty} \, ,
\end{equation}
where $\mathbb{N}_{0}$ stands for the non-negative integers, $\mathbb{N}_{0}^{(\mathbb{N})}$ denotes the set of multi-indices on $\mathbb{N}_{0}$ (i.e. finite sequences
of elements of $\mathbb{N}_{0}$) and $\mathbb{D}$ is the open unit disc of $\mathbb{C}$.
In \cite{Hi09} (see also \cite[page~65]{Hilbert_Gesam_3}) he gave the following criterion for a formal power series $\sum_{\alpha} c_{\alpha} z^{\alpha}$ to generate such a function, i.e.
to converge (absolutely) at each
point of $\mathbb{D}^{\infty}$: Every $k$-dimensional section $\sum_{\alpha\in \mathbb{N}_{0}^{k}}
c_{\alpha} z^{\alpha}$ of the series is pointwise convergent on $\mathbb{D}^{k}$,
and moreover
\begin{equation}\label{e}
\sup_{k \in \mathbb{N}} \sup_{z \in \mathbb{D}^{k}} \Big|\sum_{\alpha \in \mathbb{N}_{0}^{k}} c_{\alpha} z^{\alpha}\Big| < \infty\;.
\end{equation}
But this criterion is not correct as was later discovered by Toeplitz (see below). This fact in infinite
dimensions produces a sort of  dilemma: There is no way to develop a complex analysis of functions in infinitely many  variables which simultaneously
handles phenomena on differentiability and analyticity (as happens  in finite dimensions) . Let us explain why.

 Today a holomorphic function $f:\mathbb{D}^{\infty} \rightarrow \mathbb{C}$
is nothing else than a Fr\'{e}chet complex-differentiable function  $f:\mathbb{D}^{\infty} \rightarrow \mathbb{C}$, i.e. a complex-differentiable $\mathbb{C}$-valued function
defined on  $\mathbb{D}^{\infty}$ (we use this symbol to denote the open unit ball  of the Banach space $\ell_{\infty}$ of all bounded scalar sequences).
As usual the Banach space of all bounded holomorphic $f: \mathbb{D}^{\infty}\rightarrow \mathbb{C}$
endowed with the supremum norm will be denoted by $H_\infty(\mathbb{D}^{\infty})$.
Important examples of such functions are bounded $m$-homogeneous polynomials
$P: \ell_{\infty} \rightarrow \mathbb{C}$, restrictions of bounded $m$-linear forms
on $\ell_{\infty} \times \cdots \times \ell_{\infty}$ to the diagonal. The Banach space of all such $P$ is denoted by
$\mathcal{P}(^m \ell_{\infty})$.

It is well known that  every holomorphic $\mathbb{C}$-valued mapping $f$ on the $k$-dimensional polydisc $\mathbb{D}^{k}$ has a
monomial (or power) series expansion which converges to $f$ at every point of $\mathbb{D}^{k}$.
More precisely, for every such $f$ there is a unique family
$(c_{\alpha}(f))_{\alpha\in \mathbb{N}_{0}^{k}}$ in $\mathbb{C}$ such that $f (z) =  \sum_{\alpha \in \mathbb{N}_{0}^{k}} c_{\alpha}(f) z^{\alpha}$
for every $z \in \mathbb{D}^{k}$. The coefficients can be calculated by the Cauchy integral formula
\begin{equation} \label{monomialcoefficients}
c_{\alpha}(f) = \frac{\partial^{\alpha} f (0)}{\alpha !} = \frac{1}{(2 \pi i)^{k}} \int_{|z_{1}| = r}\ldots \int_{|z_k| = r}
\frac{f (z)}{z^{\alpha + 1}} dz_{1} \ldots dz_k \,;
\end{equation}
where $0 < r < 1$ is arbitrary. Clearly, every
holomorphic function $f:\mathbb{D}^{\infty} \rightarrow \mathbb{C}$ in infinitely many variables, if restricted
to a finite dimensional section $\mathbb{D}^{k} \times \{0\}$ that we identify with $\mathbb{D}^{k}$, has  an everywhere convergent  power
series expansion $\sum_{\alpha \in \mathbb{N}_{0}^{k}} c^{(k)}_{\alpha}(f)
z^{\alpha}$, $z \in \mathbb{D}^{k}$. And
from the Cauchy formula \eqref{monomialcoefficients} we can see that
$c_{\alpha}^{(k)}(f) = c_{\alpha}^{(k+1)}(f)$ for $\alpha \in \mathbb{N}_{0}^{k} \subseteq \mathbb{N}_{0}^{k+1}$. Thus again there is a unique family
$(c_{\alpha}(f))_{\alpha \in \mathbb{N}_{0}^{(\mathbb{N})}}$ in $\mathbb{C}$ such that at least for all $k\in \mathbb{N}$
and all  $z \in \mathbb{D}^{k}$
\[
 f (z) = \sum_{\alpha\in \mathbb{N}_{0}^{(\mathbb{N})}} c_{\alpha}(f) z^{\alpha}\,.
\]
This power series  is called the monomial series expansion of $f$, and $c_{\alpha}= c_{\alpha}(f)$ are its monomial coefficients; they satisfy
\eqref{monomialcoefficients} whenever $\alpha\in \mathbb{N}_{0}^{k}$.

At first one could expect that each bounded holomorphic function on $\mathbb{D}^{\infty}$ has a monomial series expansion which again converges at every point
and represents the function
(this is the statement of the   Hilbert's criterion mentioned above) but this is not the case: just take a non-zero functional on $\ell_{\infty}$
that is $0$ on $c_{0}$ (the space of null sequences); its monomial series expansion is $0$ on $\ell_{\infty}$ and clearly does not represent the function. One
could then try with $\mathbb{D}^{\infty}_{0}$ (the open unit ball of $c_{0}$).
Note first that a simple extension argument (see e.g. \cite[Lemma~2.2]{DeGaMa10})
allows to identify all formal power series  satisfying \eqref{e}  with  all bounded holomorphic functions on $\mathbb{D}^{\infty}_{0}$; more precisely,
each $f \in H_\infty(\mathbb{D}^{\infty}_{0})$ has a monomial series expansion as in \eqref{e}, and conversely each
power series satisfying \eqref{e} gives rise to  a unique $f \in H_\infty(\mathbb{D}^{\infty}_{0})$ for which  $c_\alpha = c_\alpha(f)$
for all $\alpha$. This is the reason why Hilbert's criterion is not correct (even if $\mathbb{D}^{\infty}$ is replaced by $\mathbb{D}_{0}^{\infty}$):
\eqref{e} does not imply \eqref{monexp} since by an example of
Toeplitz from \cite{To13} we have
\begin{align} \label{Toeppi}
\exists P \in \mathcal{P}(^2 c_{0})\, \text{ s.t. }
\forall \, \varepsilon >0 \,
\exists\, x \in \ell_{4+\varepsilon} \,:
\,\sum_\alpha |c_\alpha(P) x^{\alpha}| = \infty \, .
\end{align}
This means that there are functions $f \in H_\infty(\mathbb{D}^{\infty}_{0})$ that cannot be pointwise described by its
monomial series expansion as in \eqref{monexp} which at first glance seems scandalous. The main purpose of this article is to give concrete descriptions of the
\textit{set of monomial convergence} of all bounded holomorphic functions on $\mathbb{D}^{\infty}$:
\[
\mon H_\infty(\mathbb{D}^{\infty}) =
\Big\{z\in D^\mathbb{N} \; \big|\, \forall \, f\in H_\infty(\mathbb{D}^{\infty} ) :\; f(z) = \sum_{\alpha\in \mathbb{N}_{0}^{(\mathbb{N})}}
c_\alpha(f)z^{\alpha} \Big \}\,
\]
(where the equality means that the series converges absolutely as a net and coincides with the function) and the set of monomial convergence of all $m$-homogeneous polynomials on $\ell_{\infty}$
\[
\mon \mathcal{P}(^m\ell_{\infty}) =
\Big\{
z\in \ell_{\infty} \; \big|\, \forall P \in \mathcal{P}(^m\ell_{\infty}):\; P (z) = \sum_{\alpha\in \mathbb{N}_{0}^{(\mathbb{N})}} c_\alpha(P) z^{\alpha} \Big\}\,.
\]
Davie and Gamelin showed \cite[Theorem~5]{DaGa89} that every function in $H_\infty(\mathbb{D}^{\infty}_{0})$ can be extended to a function in
$H_\infty(\mathbb{D}^{\infty})$ with the same norm. Using this it can be seen (see e.g. \cite[Remark~6.4]{DeMaPr09}) that
\begin{equation} \label{00}
  \mon H_\infty(\mathbb{D}^{\infty}) = \mon H_\infty(\mathbb{D}^{\infty}_{0})
\text{ and }
\mon \mathcal{P}(^m\ell_{\infty})= \mon \mathcal{P}(^{m} c_{0}) \, .
\end{equation}
Let us collect and comment the results on such sets of convergence known so far. Bohr \cite{Bo13_Goett} proved
\begin{align}\label{11}
\ell_{2} \cap \mathbb{D}^{\infty}
 \subseteq \mon H_{\infty}(\mathbb{D}^{\infty})\,,
\end{align}
and Bohnenblust-Hille in \cite{BoHi31}
\begin{align} \label{22}
\ell_{\frac{2m}{m-1}}
 \subseteq \mon \mathcal{P}(^m\ell_{\infty}).
\end{align}
Moreover,  these two results in a certain sense  are optimal; to see this define
 \[
M:=\sup\big\{1 \leq p \leq \infty \, |\, \ell_{p} \cap  \mathbb{D}^{\infty} \subseteq \mon H_\infty(\mathbb{D}^{\infty})\big\} \,,
\]
 as well as for $m \in \mathbb{N}$
 \[
M_m:=\sup\big\{1 \leq p \leq \infty \, |\, \ell_{p} \subseteq \mon \mathcal{P}(^m\ell_{\infty})\big\}\,.
\]
These are two quantities which measure the size of both sets of convergence in terms of the largest possible
slices $\ell_{p} \cap \mathbb{D}^{\infty} $ included in them.
The definition of $M$ (at least implicitly) appears in \cite{Bo13_Goett}, and
\eqref{11} of course gives that $M \geq 2$. The idea of graduating $M$ through $M_m$ appears first in Toeplitz' article \cite{To13}; clearly the estimate
$4 \leq M_2$ is a reformulation of \eqref{Toeppi}. After Bohr's paper \cite{Bo13_Goett} the intensive search for the exact value of $M$ and $M_m$ was not
succesful for more then 15 years. The final answer was given by Bohnenblust and Hille in \cite{BoHi31}, who were able to prove that
\begin{equation} \label{33}
 M_m = \frac{2m}{m-1}  \,\,\,\,\,\,\, \text{ and } \,\,\,\,\,\,\, M = \frac{1}{2} \,.
\end{equation}
Their original proofs of the upper bounds are clever and ingenious. However, using modern techniques of probabilistic nature, different from the original ones, they were improved in
\cite[Example~4.9 and Example~4.6]{DeMaPr09}:
\begin{align} \label{crelle}
\ell_{2} \,\cap\, \mathbb{D}^{\infty}
 \subseteq \mon H_{\infty}(\mathbb{D}^{\infty}) \,\subseteq\,
 \bigcap_{\varepsilon>0} \ell_{2+\varepsilon},
\end{align}
and
\begin{align} \label{gadafi}
\mon \mathcal{P}(^m\ell_{\infty}) \subseteq \ell_{\frac{2m}{m-1}, \infty}.
\end{align}

So the question remains whether it is possible to ``squeeze'' our two sets of convergence in a more drastic way.
Historically all these results on
sets of monomial convergence (at least those of \eqref{Toeppi},
\eqref{11},\eqref{22}, and \eqref{33})
were  motivated through the theory of Dirichlet series.
An ordinary Dirichlet series is a series of the form $D=\sum_n a_n n^{-s}$, where the $a_n$ are complex coefficients and $s$ is a complex variable. Maximal domains
where such Dirichlet series converge conditionally, uniformly or absolutely are half planes $[\re  > \sigma]$, where $\sigma=\sigma_c,\sigma_u$ or $\sigma_a$ are called the abscissa of conditional, uniform or absolute convergence, respectively. More precisely,
$\sigma_\alpha (D)$ is the infimum of all $r\in\mathbb{R}$ such that  on $[\re  >r]$ we have convergence of $D$ of the requested type $\alpha = c,u$ or $a$.
Each Dirichlet series $D$ defines  a holomorphic  function $d: [\re  > \sigma_c] \rightarrow \mathbb{C}$. If $\sigma_b(D)$ denotes the abscissa of boundedness, i.e. the
infimum of all $r\in\mathbb{R}$ such that  $d$ on the half plane $[\re  >r]$ is bounded, then one of the fundamental theorems of Bohr from \cite{Bo13} is
\begin{align}\label{Bo13}
\sigma_u(D) = \sigma_b(D)\, .
\end{align}
Bohr's so called  \textit{absolute convergence problem} from \cite{Bo13_Goett} asked for the largest possible width of the strip in $\mathbb{C}$ on which a Dirichlet series may
converge uniformly but not absolutely. In other terms, Bohr defined the number $S := \sup_D \sigma_a(D)-\sigma_u(D)$, where
the supremum is taken over all possible Dirichlet series $D$, and asked for its precise value.
In order to explain (in modern terms) Bohr's strategy to attack the problem we denote by $\mathfrak{P}$  the vector space  of all formal power series
$\sum_{\alpha} c_{\alpha} z^{\alpha}$,
and let $\mathfrak{D}$ be the vector space  of all Dirichlet series $\sum a_{n} n^{-s}$. We denote by $(p_{n})_{n}$ the sequence of prime numbers and
$n=p_{1}^{\alpha_{1}} \cdots p_{k}^{\alpha_{k}} = p^{\alpha}$ the unique prime decomposition of $n \in \mathbb{N}$; then the linear bijection:
\begin{align} \label{vision}
\mathfrak{B} : \mathfrak{P} & \longrightarrow \mathfrak{D} \,\,, \,\,\,\,\,\,
\textstyle\sum_{\alpha \in \mathbb{N}_{0}^{(\mathbb{N})}} c_{\alpha} z^{\alpha}
\rightsquigarrow \textstyle\sum_{n=1}^{\infty} a_{n} n^{-s}\,
\end{align}
(where $a_{p^{\alpha}} = c_{\alpha}$) will be called Bohr mapping. Since every holomorphic function $f$ on $\mathbb{D}^{\infty}_{0}$ has a unique monomial
series expansion, $H_{\infty}(\mathbb{D}^{\infty}_{0})$ can be considered as a subspace of $\mathfrak{P}$. In order to see the image of
$H_{\infty}(\mathbb{D}^{\infty}_{0})$ under the Bohr mapping define the following space $ \mathcal{H}_{\infty}$ of all Dirichlet series
$\sum a_{n} n^{-s}$ which have a limit   function $d$ that is  defined and bounded on $[\re > 0]$, and note that
$ \mathcal{H}_{\infty}$  together with  the norm $\Vert \sum a_{n} n^{-s}  \Vert: = \sup_{\re s >0} \vert \sum_{n} a_{n} 1/n^{s} \vert$ forms a Banach space.
Now the following fact, essentially due to Bohr (see also \cite[Lemma~2.3 and Theorem~3.1]{HeLiSe97}), is fundamental: $\mathfrak{B}$ induces a
bijective isometry from $H_{\infty}(\mathbb{D}^{\infty}_{0})$ onto $\mathcal{H}_{\infty}$\,,
\begin{equation} \label{fund}
  H_{\infty}(\mathbb{D}^{\infty}_{0})   =  \mathcal{H}_{\infty}\,.
\end{equation}
Clearly $m$-homogeneous polynomials are mapped to Dirichlet series  $\sum a_n n^{-s}$ in $\mathcal{H}_{\infty}$ for which $a_n=0$
for those $n$ that do not have precisely $m$ prime divisors (counted according to their multiplicity); such series are also called $m$-homogeneous. Then
the restriction of  $\mathfrak{B}$ defines an isometric, onto linear homomorphism between $\mathcal{P} (^{m} c_{0})$ and $\mathcal{H}_{\infty}^{m}$
(the subspace of $\mathcal{H}_{\infty}$ consisting of $m$-homogeneous Dirichlet series):
\begin{equation} \label{fund polin}
  \mathcal{P} (^{m} c_{0}) = \mathcal{H}_{\infty}^{m} \, .
\end{equation}
Using the prime number theorem Bohr in \cite{Bo13_Goett} proved that $S=\frac{1}{M}\,,$
and  concluded from \eqref{11} that $S \leq 1/2$. Shortly after that Toeplitz with his result from \eqref{Toeppi} got   $1/4 \leq S \leq 1/2$.
Although the general theory of Dirichlet series  during the first decades of the last century was one of the most  fashionable   topics in analysis (with Bohr's  absolute
convergence problem   very much in its focus),
the question whether or not $S=1/2$ remained open for a long period.
Finally,  Bohnenblust and Hille \cite{BoHi31} in 1931 in a rather ingenious fashion answered the  problem in the positive. They proved $\eqref{33}$, and got as a
consequence what we now call the Bohr-Bohnenblust-Hille theorem: $S=\frac{1}{2}$.
One of the crucial ideas in the Bohnenblust-Hille approach is that they graduate Bohr's problem: They (at least implicitly) observe that
$S_m=\frac{1}{M_m}$\,,
where
$S_m=\sup \sigma_a(D) - \sigma_u(D)\,,$
the infimum now taken over all $m$-homogeneous Dirichlet series. This allows to deduce from $\eqref{33}$ the lower bound
$ \frac{m-1}{2m} = S_m \leq S$, and hence in the limit case  as desired $ \frac{1}{2} \leq  S$.

We finally briefly summarize the two main theorems of this article. But before, we indicate that recently some  deep new results (and techniques)
 within the Bohr-Bohnenblust-Hille cycle of ideas suggest that  more precise descriptions of $\mon H_\infty(\mathbb{D}^{\infty})$
as well as $\mon \mathcal{P}(^m c_{0})$ should be possible.
From \eqref{Bo13} it can be easily deduced that the fact $S= \frac{1}{2}$ is equivalent to $\sup_{D \in \mathcal{H}_{\infty}} \sigma_{a}(D) =\frac{1}{2}$, i.e.
for each $\varepsilon >0$ and each Dirichlet series $\sum a_n n^{-s}$ in $\mathcal{H}_\infty$ we have that
$\sum_n |a_n| \frac{1}{n^{\frac{1}{2}+ \varepsilon}} < \infty\,,$
and moreover $\frac{1}{2}$ here can not be improved. What about $\varepsilon =0$\,? The answer is yes: It was recently proved in \cite{DeFrOrOuSe11} that the
supremum of all $c \in \mathbb{R}$ such that for every $\sum a_n n^{-s} \in \mathcal{H}_\infty$
\begin{equation} \label{oje}
\sum_{n=1}^\infty |a_n|\,
\frac{
e^{ c\sqrt{\log n\log \log n}}}{n^\frac{1}{2}}
<\infty \,,
\end{equation}
equals $1/\sqrt{2}$. This is just the final step in a long series of results due to (among others) Balasubramanian, Calado, de la  Bret\'{e}che, Konyagin or Queff\'{e}lec
\cite{BaCaQu06,Br08,KoQu01,Qu95}. An interesting consequence is that each  Dirichlet series $\sum a_n n^{-s} \in \mathcal{H}_{\infty}$ even
converges absolutely on the vertical line $[\re =1/2]$. In view of Bohr's  mapping \eqref{vision} we see that the sequence $\big( p_{n}^{-\frac{1}{2}}\big)_{n}$ belongs to
$\mon H_\infty(\mathbb{D}^{\infty})$. This sequence is not contained in $\ell_{2}$ since,  due to the prime number theorem, it up to constants equals
$\big((n \log n)^{-\frac{1}{2}}\big)$. It seems that this sequence %$p^{-\frac{1}{2}}$
is the very first known example which really distinguishes $\ell_{2} \cap \mathbb{D}^{\infty}$ and $\mon H_\infty(\mathbb{D}^{\infty})$. We define the set
\[
\mathbf{B} = \Big\{ z \in \mathbb{D}^{\infty} \colon \limsup  \frac{1}{\log n} \sum_{j=1}^{n} z^{* 2}_{j}  < 1 \Big\}\,;
\]
here $z^{*}$ stands for the decreasing rearrangement of $z$ (see below for a full definition). Then our main result is Theorem~\ref{Leonhard} which states
\begin{equation} \label{reform}
\frac{1}{\sqrt{2}} \mathbf{B} \,\subseteq \,\mon H_{\infty}(\mathbb{D}^{\infty}) \,\subseteq\, \mathbf{\overline{B}}\,,
\end{equation}
improving the results from \eqref{11}, \eqref{33}, and \eqref{crelle}. Its  homogeneous counterpart is even more satisfying -- in Theorem~\ref{polinomios}
we prove that the upper inclusion in \eqref{gadafi} is optimal:
\begin{align*}
 \mon \mathcal{P}(^m\ell_{\infty}) = \ell_{\frac{2m}{m-1}, \infty}\,.
\end{align*}
Here our proof heavily depends on the following  recent homogeneous counterpart of \eqref{oje}
due to Balasubramanian, Calado and Queff\'{e}lec \cite[Theorem~1.4]{BaCaQu06}: For each $m$ there exists a constant $A_{m}$ such that for every $\sum a_{n} n^{-s} \in \mathcal{H}_{\infty}^{m}$ we have
\begin{equation} \label{ojeje}
 \sum_{n} \vert a_{n} \vert \frac{(\log n)^{\frac{m-1}{2}}}{n^{\frac{m-1}{2m}}}
\leq A_{m} \sup_{t \in \mathbb{R}} \Big\vert \sum_{n} a_{n} n^{it} \Big\vert \, .
\end{equation}
and again the parameter $\frac{m-1}{2}$  is optimal \cite[Theorem~3.1]{MaQu10}. \\

The Banach space $H_\infty(\mathbb{D}^{\infty}_{0})$ can be isometrically  identified with the Banach space $H_\infty(\mathbb{T}^{\infty})$ of  all $L_\infty$-functions $f:\mathbb{T}^{\infty} \rightarrow\mathbb{C}$  with Fourier coefficients $\hat{f}(\alpha)=0$
for $\alpha \in \mathbb{Z}^{(\mathbb{N})} \setminus \mathbb{N}_{0}^{(\mathbb{N})}$ (here $\mathbb{T}$ denotes the torus, the unit circle of $\mathbb{C}$,
and $\mathbb{T}^{\infty}$ the countable cartesian product of $\mathbb{T}$).
In the last section we prove analogs of the results we obtained for $\mon H_\infty(\mathbb{D}^{\infty})$  and $\mon \mathcal{P}(^m\ell_{\infty})$ within Hardy spaces
$H_p(\mathbb{T}^{\infty})$ and
 $H^m_p(\mathbb{T}^{\infty})$, $1 \leq p < \infty$, of functions and polynomials in infinitely many variables.
We extend and complement results of Cole and Gamelin from \cite{CoGa86}. Our main result in this section are Theorems~\ref{general} and \ref{final}: for every $1 \leq p < \infty$ we have
\[
\mon H_{p}  (\mathbb{T}^{\infty}) = \mon H_{p}^{m}  (\mathbb{T}^{\infty}) = \ell_{2} \cap \mathbb{D}^{\infty} \, .
\]

\bigskip

Let us now fix some more notation and recall some basic definitions. The set of non-negative integers is denoted by $\mathbb{N}_{0}$ and $\mathbb{D}$
and $\mathbb{T}$ respectively denote the open unit disc and circle of $\mathbb{C}$. Following \cite{Ru69} and \cite{Wo91}
$m_{k}$ and $m$ will denote the product of the normalized Lebesgue measure
respectively on $\mathbb{T}^{k}$ and  $\mathbb{T}^{\infty}$ (i.e. the unique rotation invariant Haar measures).
Given a set $\Gamma \subseteq \mathbb{C}$ we write $\Gamma^{(\mathbb{N})} = \bigcup_{k=1}^{\infty} \Gamma^{k}$, where we identify $\Gamma^{k}$
with $\Gamma^{k} \times \{ 0 \}$, (i.e. $\Gamma^{(\mathbb{N})}$ consists of sequences in $\Gamma$ that eventually vanish).
The spaces of $q$-summable sequences ($1 \leq q < \infty$) are denoted by $\ell_{q}$, while $\ell_{\infty}$ and $c_{0}$ are respectively the spaces
of bounded and null sequences. Given $z \in \ell_{\infty}$, its decreasing rearrangement is defined by
\[
z^{*}_{n} := \inf \{ \sup_{j \in \mathbb{N} \setminus J } | z_{j} |  \colon  J \subseteq \mathbb{N} \  , \ \card (J) < n \} \, .
\]
The Lorentz space $\ell_{q, \infty}$ with $1 \leq q < \infty$ consists of those sequences such that $\sup_{n} z^{*}_{n} n^{1/q} < \infty$
(and this supremum defines the norm). It is a well known fact that $\ell_{q, \infty} \subseteq c_{0}$, hence $z^{*}=(|z_{\sigma(n)}|)$ where $\sigma \mathbb{N}\to \mathbb{N}$ is an adequate permutation.
On the other hand, $\ell_{\infty}^{k}$ stands for $\mathbb{C}^{k}$ with the sup norm. \\
Given $k,m \in \mathbb{N}$ we consider the following sets of indices
\begin{gather*}
 \mathcal{M} (m,k) = \{\mathbf{j} = (j_{1}, \dots , j_{m}) \colon 1 \leq j_{1}, \dots , j_{m} \leq k \} = \{1, \ldots , k \}^{m} \\
 \mathcal{J} (m,k) = \{\mathbf{j} \in \mathcal{M} (m,k) \colon 1 \leq j_{1} \leq \dots \leq j_{m} \leq k \} \\
\Lambda (m,k) = \{ \alpha \in \mathbb{N}_{0}^{k} \colon \alpha_{1} + \cdots + \alpha_{k} =m \} \, .
\end{gather*}
An equivalence relation is defined in $\mathcal{M} (m,k)$ as follows: $\mathbf{i} \sim \mathbf{j}$ if there is a permutation $\sigma$ such that $i_{\sigma (r)} = j_{r}$
for all $r$. For each $\mathbf{i} \in \mathcal{M} (m,k)$ there is a unique $\mathbf{j} \in \mathcal{J} (m,k)$ such that $\mathbf{i} \sim \mathbf{j}$.
On the other hand, there is a one-to-one relation between  $\mathcal{J} (m,k)$ and $\Lambda (m,k)$: Given $\mathbf{j}$,
one can define $^{\mathbf{j}}\alpha$ by doing $^{\mathbf{j}}\alpha_{r} = | \{ k \colon j_{m}=r \}|$; conversely, for each $\alpha$, we consider
$\mathbf{j}_{\alpha} = (1, \stackrel{\alpha_{1}}{\dots} , 1, 2,\stackrel{\alpha_{2}}{\dots} ,2 ,
\dots , k \stackrel{\alpha_{k}}{\dots} ,k)$. Note that $\card[\mathbf{j}_{\alpha}] = \frac{m!}{\alpha !}$ for every $\alpha \in \Lambda (m,k)$.
Given a multi-index $\alpha \in \mathbb{N}_{0}^{k}$ we write $\vert \alpha \vert = \alpha_{1} + \cdots + \alpha_{k}$.\\
Taking this correspondence into account, the monomial series expansion of a polynomial $P \in \mathcal{P} (^{m} \ell_{\infty}^{k})$ can be expressed in different ways
(we write $c_{\alpha} = c_{\alpha}(P)$)
\begin{equation} \label{polin varios}
  \sum_{\alpha \in \Lambda(m,k)} c_{\alpha} z^{\alpha} = \sum_{\mathbf{j} \in \mathcal{J}(m,k)} c_{\mathbf{j}} z_{\mathbf{j}}
= \sum_{1 \leq j_{1} \leq \ldots \leq j_{m} \leq k} c_{j_{1} \ldots j_{m}} z_{j_{1}} \cdots z_{j_{m}} \, .
\end{equation}
Following standard notation for each $n \in \mathbb{N}$ we write $\Omega (n)=\vert \alpha \vert$ whenever $n = p^{\alpha}$ (this counts the prime divisors of
$n$, according to their multiplicity). Then
a Dirichlet series $\sum a_{n} n^{-s}$ is called $m$-homogeneous if $a_{n}=0$ for every $\Omega (n) \neq m$. We denote by $\mathcal{H}_{\infty}^{m}$ for the
space of $m$-homogeneous Dirichlet series in $\mathcal{H}_{\infty}$.\\
Finally, the norms for $m$-homogeneous polynomials and $m$-linear forms on $\ell_{\infty}$ are as usual defined by
$\Vert P \Vert = \sup_{z \in \mathbb{D}^{\infty}} \vert P(z) \vert$ and $\Vert A \Vert = \sup_{z_{j} \in \mathbb{D}^{\infty}} \vert A(z_{1} , \ldots , z_{m}) \vert$.\\

If $f \in H_{\infty}(\mathbb{D}^{\infty})$ and $\sigma$ is a permutation then the function $f_{\sigma}$ defined by $f_{\sigma}(z)$ $= f ((z_{\sigma (n)})_{n})$ is
again in $H_{\infty}(\mathbb{D}^{\infty})$. This implies (see e.g. \cite[page~550]{DeGaMaPG08}) that if $z \in \mon H_{\infty}(\mathbb{D}^{\infty})$
then every permutation of $z$ is again in $\mon H_{\infty}(\mathbb{D}^{\infty})$. The same happens for $\mon \mathcal{P} (^{m} \ell_{\infty})$.\\
On the other hand we know from \cite[page~30]{DeMaPr09} that  if $z$ belongs to $\mon H_{\infty}(\mathbb{D}^{\infty})$ (or to $\mon \mathcal{P} (^{m} \ell_{\infty})$)
then it is in $c_{0}$. Hence its decreasing rearrangement is a permutation of $z$. Thus $z \in \mon H_{\infty}(\mathbb{D}^{\infty})$ if and only if
$z^{*} \in \mon H_{\infty}(\mathbb{D}^{\infty})$ (and the same for $\mon \mathcal{P} (^{m} \ell_{\infty})$).

\section{Homogeneous polynomials}
By \eqref{fund polin} there is a bijection between $\mathcal{P} (^{m} c_{0})$ and $\mathcal{H}_{\infty}^{m}$.
We know now thanks to \eqref{ojeje} the precise behaviour in the side of Dirichlet
series. Our aim is to transfer this knowledge to the polynomials side in order to get a better understanding. We do that in the following result.

\begin{theorem} \label{polinomios}
For each $m \in \mathbb{N}$ we have
$\mon \mathcal{P}(^{m} \ell_{\infty}) = \ell_{\frac{2m}{m-1}, \infty}$. Moreover, there exists a constant $C>0$ such that
\begin{equation}\label{eq:polinomios}
 \sum_{\vert \alpha \vert =m} \vert c_{\alpha}(P) z^{\alpha} \vert
\leq C^{m} \Vert z \Vert^{m} \Vert P \Vert \, .
\end{equation}
for every $z \in \ell_{\frac{2m}{m-1}, \infty}$ and every $P \in  \mathcal{P}(^{m} \ell_{\infty})$.
\end{theorem}
The inequality \eqref{ojeje} is enough to get the set equality $\mon \mathcal{P}(^{m} \ell_{\infty}) = \ell_{\frac{2m}{m-1}, \infty}$; however, since we are also interested in the
behaviour of the norms \eqref{eq:polinomios} we need the following refinement of \cite[Theorem~1.4]{BaCaQu06}:
\begin{lemma} \label{queffelec}
There is a constant $K>0$ such that for every $m$-homogeneous Dirichlet polynomial $\sum a_{n} n^{-s} \in \mathcal{H}_{\infty}^{m}$,
\[
\sum_{n} \vert a_{n} \vert \frac{(\log n)^{\frac{m-1}{2}}}{n^{\frac{m-1}{2m}}}
\leq K^{m} \sup_{t \in \mathbb{R}} \Big\vert \sum_{n} a_{n} n^{it} \Big\vert \, .
\]
\end{lemma}

To prove this we need the following Lemma, a variant of one of the two main ingredients of the proof of the hypercontractivity of the Bohnenblust-Hille inequality
\cite[(14)]{DeFrOrOuSe11} (see \cite[Lemma 5.2]{DeScSe} for a more general setting).
\begin{lemma}  \label{lemma2}
For every $k \in \mathbb{N}$ and every $m$-homogeneous polynomial $P \in \mathcal{P}(^{m} \ell_{\infty}^{k})$
with coefficients $(c_{\mathbf{j}})_{\mathbf{j} \in \mathcal{J}(m,k)}$ (we write $c_{\mathbf{j}}=c_{\mathbf{j}}(P)$) we have
\[
\sum_{j_{m} = 1}^{k} \bigg(\sum_{\substack{\mathbf{j} \in \mathcal{J}(m-1,k) \\ j_{1} \leq \ldots \leq j_{m-1} \leq j_{m}}}
\vert c_{j_{1} \ldots j_{m}} \vert^{2} \bigg)^{\frac{1}{2}} %\\
 \leq m \sqrt{2}^{m-1} \left(1+\frac{1}{{m-1}} \right)^{m-1} % \sup_{z \in \mathbb{D}^{N}} \Big\vert \sum_{\mathbf{j} \in \mathcal{J}(m,N)} c_{j_{1} \ldots j_{m}} z_{j_{1}} \dotsm z_{j_{m}} \Big\vert \, .
\Vert P \Vert \, .
\]
\end{lemma}
\begin{proof}
Let $P\in \mathcal{P}(^m \ell_{\infty}^{k})$ be an $m$-homogeneous polynomial and $A$ its associated symmetric $m$-linear mapping. It is well known that the
monomial coefficients $c_{j_{1} \ldots j_{m}}$ of $P$ and the coefficients $a_{i_{1}\ldots i_{m}}:= A(e_{i_{1}},\ldots,e_{i_{m}})$ defining $A$ satisfy
$c_{\mathbf{j}} = \card[\mathbf{j}] a_{\mathbf{j}}$.\\
On the other hand, for each $\mathbf{j} \in \mathcal{J}(m-1,k)$ and $1\leq j \leq k$ we write $(\mathbf{j}, j) = (j_{1}, \ldots , j_{m-1}, j) \in
\mathcal{M}(m,k)$ and we have
\[
 \frac{\card[(\mathbf{j}, j)]}{\card[\mathbf{j}]}
= \frac{m!}{(m-1)!} \cdot \frac{\left\vert \{r \mid   j_{r} = 1\} \right\vert !}{\left\vert \{r \mid  (\mathbf{j}, j)_r = 1\}\right\vert !}
\dotsm\frac{\left\vert \{r \mid  j_r = k\}\right\vert !}{\left\vert \{r \mid  (\mathbf{j}, j)_r = k\}\right\vert !} \leq m \, .
\]
Hence
\begin{multline*}
\sum_{j_{m}=1}^{k} \biggl(\sum_{\substack{\mathbf{j} \in \mathcal{J}(m-1,k)\\ j_{1}\leq\ldots\leq j_{m}}} \vert c_{j_{1} \ldots j_{m}} \vert^{2} \biggr)^{\frac{1}{2}}
=  \sum_{j=1}^{k} \biggl(\sum_{\substack{\mathbf{j} \in \mathcal{J}(m-1,k)\\ j_{1} \leq \ldots \leq j_{m-1} \leq j}} \vert \card[(\mathbf{j}, j)] a_{(\mathbf{j}, j)} \vert^{2} \biggr)^{\frac{1}{2}} \\
\leq \sum_{j=1}^{k} \biggl(\sum_{\mathbf{j} \in \mathcal{J}(m-1,k)} \vert \card[(\mathbf{j}, j)] a_{(\mathbf{j}, j)} \vert^2 \biggr)^{\frac{1}{2}} \\
\leq m \sum_{j=1}^{k} \biggl(\sum_{\mathbf{j} \in \mathcal{J}(m-1,k)} \vert \card[\mathbf{j}] a_{(\mathbf{j}, j)} \vert^{2} \biggr)^{\frac{1}{2}} \,.
\end{multline*}
In the proof of \cite[Theorem~9]{Ba02} (see also \cite[Lemma~3.6]{DeFr11}) it
is proved that for every polynomial $P \in \mathcal{P}(^{m} \mathbb{C}^{k})$ with $P(z)=\sum_{\alpha\in\Lambda(m,k)} c_{\alpha}z^{\alpha}$ we have
\begin{equation} \label{eq:bonami 1}
\Bigg(\sum_{\alpha\in\Lambda(m,k)} \left\vert c_{\alpha} \right\vert^{2} \Bigg)^{\frac{1}{2}}
\leq \sqrt{2}^{m} \int_{\mathbb{T}^{k}} \biggl\vert \sum_{\alpha\in\Lambda(m,k)} c_{\alpha} z^{\alpha} \biggr\vert dm_{k}(z)
\end{equation}
Using this  we get
\begin{align*}
 \sum_{j=1}^{k} & \biggl(\sum_{\mathbf{j} \in \mathcal{J}(m-1,k)} \vert \card[\mathbf{j}] a_{(\mathbf{j}, j)} \vert^{2} \biggr)^{\frac{1}{2}} \\
& \leq \sqrt{2}^{m-1} \sum_{j=1}^{k}
\int_{\mathbb{T}^{k}} \biggl\vert \sum_{\mathbf{j} \in \mathcal{J}(m-1,k)} \card[\mathbf{j}] a_{(\mathbf{j}, j)} z_{j_{1}} \dotsm z_{j_{m-1}} \biggr\vert dm_{k}(z) \\
& = \sqrt{2}^{m-1} \int_{\mathbb{T}^{k}} \sum_{j=1}^{k}
\biggl\vert \sum_{\mathbf{j}\in\mathcal{M}(m-1,k)} a_{(\mathbf{j}, j)} z_{j_{1}} \dotsm z_{j_{m-1}} \biggr\vert dm_{k}(z) \\
& \leq \sqrt{2}^{m-1} \sup_{z \in \mathbb{T}^{k}}  \sum_{j=1}^{k}
\biggl\vert \sum_{\mathbf{j}\in\mathcal{M}(m-1,k)} a_{(\mathbf{j}, j)} z_{j_{1}} \dotsm z_{j_{m-1}} \biggr\vert \\
& = \sqrt{2}^{m-1} \sup_{z \in \mathbb{D}^{k}} \sup_{y \in \mathbb{D}^{k}}
\biggl\vert \sum_{j=1}^{k}  \sum_{\mathbf{j}\in\mathcal{M}(m-1,k)} a_{(\mathbf{j}, j)} z_{j_{1}} \dotsm z_{j_{m-1}} y_{j} \biggr\vert \\
& = \sqrt{2}^{m-1} \sup_{z,y  \in \mathbb{D}^{k}}\vert A(z, \stackrel{m-1}{\ldots}, z , y) \vert
\leq \sqrt{2}^{m-1} \left(1+\frac1{m-1}\right)^{m-1} \Vert P \Vert \, ,
\end{align*}
where the last inequality follows from an estimate of Harris \cite[Theorem 1]{Ha75}. This completes the proof.
\end{proof}
A careful analysis of \cite[Theorem 1.4]{BaCaQu06} using Lemma~\ref{lemma2} gives  Lemma~\ref{queffelec} (see also \cite[Theorem 5.1]{DeScSe} for more details and a
vector valued version). We are now ready to give the proof of the main result of this section.

\begin{proof}[Proof of Theorem~\ref{polinomios}]
In view of \eqref{gadafi} we only have to show $\ell_{\frac{2m}{m-1}, \infty}$ is a subset of $\mon \mathcal{P}(^{m} \ell_{\infty})$. We
begin by observing that for every $(x_{j})_{j}$ with $x_{j} \geq 0$ and all $\alpha \in \mathbb{N}_{0}^{(\mathbb{N})}$ with $\vert \alpha \vert =m$ we have, by
a simple application of the binomial formula,
\begin{equation} \label{observacion}
(x^{\alpha})^{1/m} \leq \sum_{j}  \alpha_{j} x_{j} \,.
\end{equation}
Now, if $z \in \ell_{\frac{2m}{m-1}, \infty}$ we have $\sup_{n} z^{\ast}_{n} n^{\frac{m-1}{2m}} = \Vert z \Vert < \infty$.
Then by the  Prime Number Theorem there is a universal constant $C_{1}>0$ such that, for all $n \in \mathbb{N}$,
\begin{equation*}
z^{\ast}_{n} \leq \Vert z \Vert \frac{1}{n^{\frac{m-1}{2m}}}
\leq  \Vert z \Vert \Big( \frac{\log (n \log n) }{n \log n} \Big)^{\frac{m-1}{2m}}
\leq \Vert z \Vert C_{1} \Big( \frac{\log p_{n} }{p_{n}} \Big)^{\frac{m-1}{2m}} \, .
\end{equation*}
Now for a fixed $\alpha \in \mathbb{N}_{0}^{(\mathbb{N})}$ we have $\textstyle\sum \alpha_{j} \log  p_{j}= \sum \log  p_{j}^{\alpha_{j}}  = \log \prod p_{j}^{\alpha_{j}}
= \log p^{\alpha}$. Applying  \eqref{observacion} with $x_{j} = \log p_{j}$ we get
\begin{multline*}
 z^{\ast \alpha} \leq (\Vert z \Vert C_{1})^{m} \bigg( \Big( \frac{\log p }{p} \Big)^{\frac{m-1}{2m}} \bigg)^{\alpha}
= (\Vert z \Vert C_{1})^{m} \frac{\Big( \big[ (\log  p)^{\alpha} \big]^{1/m} \Big)^{\frac{m-1}{2}} }{(p^{\alpha})^{\frac{m-1}{2m}}}  \\
\leq (\Vert z \Vert C_{1})^{m} \frac{\Big( \sum \alpha_{j} \log  p_{j} \Big)^{\frac{m-1}{2}} }{(p^{\alpha})^{\frac{m-1}{2m}}}
= (\Vert z \Vert C_{1})^{m} \frac{( \log p^{\alpha} )^{\frac{m-1}{2}} }{(p^{\alpha})^{\frac{m-1}{2m}}} \, .
\end{multline*}
Given any polynomial $P \in \mathcal{P} (^{m} \ell_{\infty})$ with coefficients $(c_{\alpha})$ we apply this and we write $a_{n} = c_{\alpha}$ for $n=p^{\alpha}$ to obtain
\begin{align*}
 \sum_{\vert \alpha \vert =m} \vert c_{\alpha} \vert z^{\ast \alpha}
\leq (\Vert z \Vert C_{1})^{m} \sum_{\vert \alpha \vert =m} & \vert a_{p^{\alpha}}  \vert
\frac{( \log p^{\alpha} )^{\frac{m-1}{2}} }{(p^{\alpha})^{\frac{m-1}{2m}}}  \\
& =  (\Vert z \Vert C_{1})^{m} \sum_{\Omega (n) =m} \vert a_{n}  \vert
\frac{( \log n)^{\frac{m-1}{2}} }{n^{\frac{m-1}{2m}}} \, .
\end{align*}
Using Lemma~\ref{queffelec} and the fact that $\mathfrak{B} : \mathcal{P} (^{m} c_{0}) \rightarrow \mathcal{H}_{\infty}^{m}$ is an isometry (\eqref{fund polin})we have
\[
 \sum_{\vert \alpha \vert =m} \vert c_{\alpha} \vert z^{\ast \alpha}
\leq C^{m} \Vert z \Vert^{m} \sup_{t \in \mathbb{R}} \Big\vert \sum_{\Omega(n)=m} a_{n} n^{it} \Big\vert
= C^{m} \Vert z \Vert^{m}  \Vert P_{\vert_{c_{0}}} \Vert
\leq  C^{m} \Vert z \Vert^{m}  \Vert P \Vert \, .
\]
This shows that $z^{\ast} \in \mon \mathcal{P}(^{m} \ell_{\infty})$. Now, $z \in c_{0}$ since $z \in \ell_{\frac{2m}{m-1}, \infty}$, hence there is a permutation $\sigma$
such that $z_{n} = z_{\sigma (n)}^{\ast}$. From this we have $z  \in \mon \mathcal{P}(^{m} \ell_{\infty})$ and \eqref{eq:polinomios} holds.
\end{proof}

\begin{remark} \label{naipes}
Inequality~\eqref{eq:polinomios} can be rewritten as
\begin{equation} \label{rewritten}
\sup_{z \in B_{\ell_{\frac{2m}{m-1}, \infty}}(\mathbb{N})} \sum_{\vert \alpha \vert =m}  \vert c_{\alpha}(P) z^{\alpha} \vert
\leq  C^{m} \Vert P \Vert
\end{equation}
for every $P \in \mathcal{P} (^{m} \ell_{\infty})$. This can be seen as some sort of Bohnenblust-Hille like inequality in the following sense. The
Bohnenblust-Hille inequality says that
$\Vert (c_{\alpha} ) \Vert_{\ell_{\frac{2m}{m+1}}(\mathbb{N}_{0}^{(\mathbb{N})})} \leq C^{m}  \Vert P \Vert$ (note that both the Bohnenblust-Hille inequality
and \eqref{rewritten} are hypercontractive). Fournier showed in \cite{Fo87} that $\Vert (c_{\alpha} ) \Vert_{\ell_{\frac{2m}{m+1},1}(\mathbb{N}_{0}^{(\mathbb{N})})}
\leq C_{m}  \Vert P  \Vert$, which by duality can be refrased as
\[
\sup_{\lambda \in B_{\ell_{\frac{2m}{m-1}, \infty}(\mathbb{N}_{0}^{(\mathbb{N})})}} \sum_{\vert \alpha \vert =m}  \vert c_{\alpha}(P) \vert \lambda_{\alpha}
\leq  C_{m} \Vert P \Vert \, .
\]
\end{remark}

\begin{remark} \label{ele20}
Inequality~\eqref{eq:polinomios} allows us to improve \eqref{11}. Let us define $\ell_{2,0} = \{ z \in \ell_{\infty} \colon \lim_{n} z^{*}_{n} \sqrt{n} = 0 \}$;
then
\begin{equation} \label{2-0}
\ell_{2,0} \cap \mathbb{D}^{\infty} \subseteq \mon H_{\infty} (\mathbb{D}^{\infty}) \, .
\end{equation}
We sketch now a proof of this fact. Since $B_{\ell_{2, \infty}}\subseteq \bigcap_{m\in \mathbb{N}} B_{\ell_{\frac{2m}{m-1}, \infty}}$, by   \eqref{eq:polinomios} and
\cite[Theorem~5.1]{DeMaPr09}  there exists an $r>0$ such that $r B_{\ell_{2, \infty}}
\subseteq \mon H_{\infty} (\mathbb{D}^{\infty})$. Then $\big( \frac{r}{\sqrt{n}} \big)_{n} \in \mon H_{\infty} (\mathbb{D}^{\infty})$; this easily
gives (using \cite[Lemma 2]{DeGaMaPG08}) that $z^{*}$ belongs to $\mon H_{\infty} (\mathbb{D}^{\infty})$ for every $z \in \ell_{2,0} \cap \mathbb{D}^{\infty}$
and \eqref{2-0} follows. In the next section we improve this estimate, but let us note that by the  Prime Number Theorem
the sequence $\big(\frac{1}{p_{n}^{1/2}} \big) \in \ell_{2,0}$ and hence belongs to $\mon H_{\infty} (\mathbb{D}^{\infty})$ while it does not belong to
$\ell_{2}$.
\end{remark}

\section{Holomorphic functions}
Again we have the link \eqref{oje} between holomorphic functions and Dirichlet series \eqref{fund}, and again we know precisely what happens in the Dirichlet side,
while our knowledge on the power series side \eqref{crelle} is weaker. In the previous section we managed to strengthen it a little bit. We show now that we can actually
go further.

\begin{theorem} \label{Leonhard}
Let $z \in \mathbb{D}^{\infty}$.
\begin{enumerate}
 \item If \, $\displaystyle\limsup_{n \to \infty} \frac{1}{\log n} \sum_{j=1}^{n} z^{* 2}_{j} < \frac{1}{2}$, then $z \in  \mon H_{\infty}(\mathbb{D}^{\infty})$. \label{Leo 1}

\item If  $z \in  \mon H_{\infty}(\mathbb{D}^{\infty})$, then  $\displaystyle\limsup_{n \to \infty} \frac{1}{\log n} \sum_{j=1}^{n} z^{* 2}_{j} \leq 1$. \label{Leo 2}
\end{enumerate}
 \end{theorem}

\begin{remark} \label{c0}
  If $z$ is such that $\limsup_{n} \frac{1}{\log n} \sum_{j=1}^{n} z^{* 2}_{j} < \infty$, then $\sup_{n} \frac{1}{\log n} \sum_{j=1}^{n} z^{* 2}_{j} < \infty$
and, since $n z^{* 2}_{n} \leq   \sum_{j=1}^{n} z^{* 2}_{j}$, this implies
\begin{equation} \label{sup log}
\sup_{n \in \mathbb{N}}   z^{*}_{n} \sqrt{\frac{n}{\log n}} = K < \infty \, .
\end{equation}
Hence $\lim_{n}  z^{*}_{n} =0$ and $z \in c_{0}$. This in particular means that the decreasing rearrangement is just a permutation of $z$.
\end{remark}

\noindent We need the following Lemma.

\begin{lemma} \label{equivalences}
 Let $(r_{j})_{j}$ be non-increasing sequence in $[0,1[$ that converges to $0$ and $(\alpha_{n})_{n}$ a non-decreasing sequence of positive numbers that converges to $\infty$;
then the following are equivalent,
\begin{enumerate}
 \item For every $\rho < 1$ there exists $C_{\rho} > 0$ such that $\sum_{j=1}^{n} - \log \big(1 - (\rho r_{j})^{2} \big) \leq C_{\rho} + \alpha_{n}$ for all $n$.
\item For every $\rho < 1$ there exists $C_{\rho}> 0$ such that $\sum_{j=1}^{n} (\rho r_{j})^{2}  \leq C_{\rho} + \alpha_{n}$ for all $n$.
\item $\limsup_{n} \frac{1}{\alpha_{n}} \sum_{j=1}^{n} r_{j}^{2} \leq 1$.
\end{enumerate}
\end{lemma}

\begin{proof}
 \textit{(i) $\Rightarrow$ (ii)} follows from the fact that $x < - \log (1-x)$ for all $0<x<1$. On the other hand, to prove \textit{(ii) $\Rightarrow$ (i)}, let us fix $\rho < 1$ and choose some
$\rho < \rho' <1$. Since
\[
 \lim_{j \to \infty} \frac{- \log \big(1 - (\rho r_{j})^{2} \big)}{r_{j}^{2}} = \rho^{2} < \rho'^{2}
\]
there exists $n_{0}$ such that $- \log \big(1 - (\rho r_{j})^{2} \big) < (\rho' r_{j})^{2}$  for every $j \geq n_{0}$. On the one hand, for $n \leq n_{0}$ we have
\[
\sum_{j=1}^{n} - \log \big(1 - (\rho r_{j})^{2} \big) \leq   \sum_{j=1}^{n_{0}} - \log \big(1 - (\rho r_{j})^{2} \big) + C_{\rho'} + \alpha_{n}
\]
On the other hand, if $n >n_{0}$ then we have
\begin{multline*}
 \sum_{j=1}^{n} - \log \big(1 - (\rho r_{j})^{2} \big) \leq \sum_{j=1}^{n_{0}} - \log \big(1 - (\rho r_{j})^{2} \big) + \sum_{j=n_{0}+1}^{n} (\rho' r_{j})^{2} \\
\leq \sum_{j=1}^{n_{0}} - \log \big(1 - (\rho r_{j})^{2} \big) + \sum_{j=1}^{n} (\rho' r_{j})^{2}
\leq \sum_{j=1}^{n_{0}} - \log \big(1 - (\rho r_{j})^{2} \big) +  C_{\rho'} + \alpha_{n} \,.
\end{multline*}
Definig $C_{\rho} = \sum_{j=1}^{n_{0}} - \log \big(1 - (\rho r_{j})^{2} \big) +  C_{\rho'}$, this altogether shows that \textit{(i)} holds.\\
Let us assume now that \textit{(ii)} holds, then for each $\rho < 1$ we have
\[
 \frac{1}{\alpha_{n}} \sum_{j=1}^{n} r_{j}^{2} =  \frac{1}{\rho^{2}} \frac{1}{\alpha_{n}} \sum_{j=1}^{n} (\rho r_{j})^{2}
\leq \frac{1}{\rho^{2}} \frac{1}{\alpha_{n}} ( C_{\rho} + \alpha_{n} )
= \frac{1}{\rho^{2}}  \Big( \frac{C_{\rho}}{\alpha_{n}} + 1 \Big)  \, ,
\]
and this converges to $\frac{1}{\rho^{2}}$. This implies
\[
 \limsup_{n \to \infty} \frac{1}{\alpha_{n}} \sum_{j=1}^{n} r_{j}^{2} \leq 1 \, .
\]
Let us finally show that \textit{(iii)} implies \textit{(ii)}. Fix $\rho < 1$, then
\[
 \limsup_{n \to \infty} \frac{1}{\alpha_{n}} \sum_{j=1}^{n} (\rho r_{j})^{2} \leq \rho^{2} < 1 \, .
\]
Then there is some $n_{0}$ such that $\frac{1}{\alpha_{n}} \sum_{j=1}^{n} (\rho r_{j})^{2} \leq 1$ for every $n \geq n_{0}$. Hence
\[
\sum_{j=1}^{n} (\rho r_{j})^{2} \leq \alpha_{n} \leq  \sum_{j=1}^{n_{0}} (\rho r_{j})^{2} + \alpha_{n}
\]
for every $n \geq n_{0}$. On the other hand, if $n \leq n_{0}$,
\[
 \sum_{j=1}^{n} (\rho r_{j})^{2} \leq \sum_{j=1}^{n_{0}} (\rho r_{j})^{2} \leq  \sum_{j=1}^{n_{0}} (\rho r_{j})^{2} + \alpha_{n} \, .
\]
This finally shows that \textit{(ii)} holds.
\end{proof}

We can now give the \textit{Proof of Theorem~\ref{Leonhard}--(\ref{Leo 2}).}
Let us fix $z \in \mon H_{\infty}(\mathbb{D}^{\infty})$; then $z \in c_{0}$ and without loss of generality we may assume
$z=r=(r_{n})_{n}$ with $0 \leq r_{n} < 1$, non-increasing and  that converges to $0$. By \cite[Lemma~4.1]{DeMaPr09} there exists $C_{1}$ such that
\begin{equation} \label{r mon}
 \sum_{\vert \alpha \vert =m} \vert c_{\alpha}(P) \vert r^{\alpha} \leq C_{1} \Big\Vert \sum_{\vert \alpha \vert =m} c_{\alpha}(P) z^{\alpha}  \Big\Vert
\end{equation}
for every $m$-homogeneous polynomial $P \in \mathcal{P}(^{m} \ell_{\infty}^{n})$. By \cite[Chapter 6, Theorem 4]{Ka85} there
exists  $C_{2} > 0$ so that
for any finite family $(a_{\alpha})_{\alpha}$ of complex numbers with $\alpha \in \mathbb{N}_{0}^{n}$ and $\vert \alpha \vert=m$, there
exists a choice of signs $\varepsilon_{\alpha} = \pm 1$ such that
\begin{equation} \label{Kahane}
\Big\Vert \sum_{\alpha} \varepsilon_{\alpha} a_{\alpha} z^{\alpha} \Big\Vert \leq C_{2} \Big( n \sum_{\alpha} \vert a_{\alpha} \vert^{2} \log m \Big)^{\frac{1}{2}} \, .
\end{equation}
Consider $a_{\alpha} = r^{\alpha}$ for $\alpha \in \Lambda (m,n)$. We
apply  \eqref{r mon} for $c_{\alpha} = \varepsilon_{\alpha} a_{\alpha}$ and then \eqref{Kahane} to get
\begin{multline*}
 \sum_{\alpha \in  \Lambda (m,n)} r^{2 \alpha} = \sum_{\alpha \in  \Lambda (m,n)} \vert \varepsilon_{\alpha} r^{\alpha} \vert r^{\alpha}
\leq C_{1} \Big\Vert \sum_{\alpha \in  \Lambda (m,n)} \varepsilon_{\alpha} a_{\alpha} z^{\alpha} \Big\Vert \\
\leq C_{1} C_{2} \Big( n \sum_{\alpha \in  \Lambda (m,n)} \vert a_{\alpha} \vert^{2} \log m \Big)^{\frac{1}{2}}
= C \sqrt{n \log m}\Big(\sum_{\alpha \in  \Lambda (m,n)} r^{2 \alpha}\Big)^{\frac{1}{2}} \, .
\end{multline*}
This implies
\[
 \Big(\sum_{\substack{\alpha \in \mathbb{N}_{0}^{n} \\ \vert \alpha \vert=m}} r^{2 \alpha}\Big)^{\frac{1}{2}}  \leq C \sqrt{n \log m} \,.
\]
Then, for a given $\rho < 1$ we have
\begin{multline*}
 \prod_{j=1}^{n} \frac{1}{1-(\rho r_{j})^{2}} = \sum_{\alpha \in \mathbb{N}_{0}^{n}} (\rho r)^{2 \alpha}
= 1 + \sum_{m=1}^{\infty} \rho^{2m} \sum_{\substack{\alpha \in \mathbb{N}_{0}^{n} \\ \vert \alpha \vert=m}} r^{2 \alpha} \\
\leq 1 + \sum_{m=1}^{\infty} \rho^{2m} C^{2} n \log m
= 1 + C^{2} n \sum_{m=1}^{\infty} \rho^{2m} \log m \leq K n \, ,
\end{multline*}
where the last inequality holds because the series $\sum_{m} \rho^{2m} \log m$ converges. This implies
\[
 \log \Big(  \prod_{j=1}^{n} \frac{1}{1-(\rho r_{j})^{2}} \Big) \leq \log K + \log n
\]
and
\[
 \sum_{j=1}^{n} - \log \big(1-(\rho r_{j})^{2} \big) \leq C_{\rho} + \log n \, .
\]
Lemma~\ref{equivalences} gives the conclusion. \qed

We give now the \textit{Proof of Theorem~\ref{Leonhard}--(\ref{Leo 1}).} By Remark~\ref{c0} $z$ is a null sequence and its decreasing rearrangement $z^{*}$ can be
written as a sequence $(r_{n})_{n}$ of non-negative numbers decreasing to $0$. We write
\[
\frac{1}{a^{2}}= \limsup_{n \to \infty} \frac{1}{\log n} \sum_{j=1}^{n} r_{j}^{2}  < \frac{1}{2}
\]
Without loss of generality we may assume that $a r_{j} < 1$ for every $j$. Indeed, if that were not the case, take $n_{0}$ such that $r_{n} < \frac{1}{a}$ for
all $n \geq n_{0}$ and define
\[
 \tilde{r}_{n}
= \begin{cases}
   r_{n_{0}} & \text{if } n < n_{0} \\
   r_{n} & \text{if } n_{0} \leq n \, .
  \end{cases}
\]
Then the sequence $(\tilde{r}_{n})_{n}$ is non-increasing and $0 < \tilde{r}_{n} < 1$ for all $n$; moreover $\limsup \frac{1}{\log n} \sum_{j=1}^{n} \tilde{r}_{j}^{2}
=\frac{1}{a^{2}} < \frac{1}{2}$ and $a \tilde{r}_{n} <1$ for all $n$. By  \cite[Lemma 2]{DeGaMaPG08}  $r \in \mon H_{\infty}(\mathbb{D}^{\infty})$ if and
only if $\tilde{r} \in \mon H_{\infty}(\mathbb{D}^{\infty})$. \\
From \eqref{sup log} we have $\log \frac{1}{r_{n}^{2}} \geq \log \frac{n}{\log n} - \log K$. But
$\lim_{n} \frac{\log n }{\log \frac{n}{\log n}} = 1$; hence
\[
 \limsup_{n \to \infty} \frac{1}{\log \frac{1}{r_{n}^{2}}} \sum_{j=1}^{n} r_{j}^{2} \leq
\limsup_{n \to \infty} \frac{1}{\log n} \sum_{j=1}^{n} r_{j}^{2} = \frac{1}{a^{2}} \leq 1
\]
and this gives $\limsup_{n} \frac{1}{\log \frac{1}{r_{n}^{2}}} \sum_{j=1}^{n} (a r_{j})^{2} \leq 1$.
We fix now $\sqrt{2} < a_{0} < a$ and take $\rho_{0} < 1$ with $\rho_{0} a = a_{0}$. By applying
Lemma~\ref{equivalences} to the sequence $(a r_{j})_{j}$ we find some constant $C_{0}>0$ such that, for all $n$
\[
 \sum_{j=1}^{n} - \log \big( 1 - (a_{0} r_{j} )^{2} \big) \leq C_{0} + \log \frac{1}{r_{n}^{2}} \, .
\]
Since all the terms are positive we have for all $n$, $m$
\begin{multline*}
 \sum_{\substack{\alpha \in \mathbb{N}_{0}^{n} \\ \vert \alpha \vert =m}} ( a_{0} r)^{2 \alpha}
= \sum_{\alpha \in \mathbb{N}_{0}^{n}} ( a_{0} r)^{2 \alpha}
=\prod_{j=1}^{n} \frac{1}{1 - (a_{0} r_{j} )^{2}}
= \prod_{j=1}^{n} e^{- \log \big( 1 - (a_{0} r_{j} )^{2} \big)} \\
= e^{\sum_{j=1}^{n} - \log \big( 1 - (a_{0} r_{j} )^{2} \big)}
\leq e^{C_{0}} e^{\log \frac{1}{r_{n}^{2}}} = C \frac{1}{r_{n}^{2}} \,.
\end{multline*}
This implies, for every $m$,
\begin{equation} \label{cota a}
 \sup_{n \in \mathbb{N}} r_{n}^{2}  \sum_{\substack{\alpha \in \mathbb{N}_{0}^{n} \\ \vert \alpha \vert =m}} r^{2 \alpha} \leq \frac{C}{a_{0}^{2m}} \, .
\end{equation}
We take now $f \in H_{\infty} (\mathbb{D}^{\infty})$ with monomial coefficients $(c_{\alpha}(f))_{\alpha}$ and Taylor expansion $f= \sum_{m} P_{m}$. For each of these polynomials
$P_{m}$ we consider $P_{m,n}$, the restriction to $\ell_{\infty}^{n}$.
The Cauchy inequalities give $\Vert P_{m,n} \Vert \leq \Vert P_{m} \Vert \leq \Vert f \Vert$ for all $m,n$.
Then for each $m \geq 1$ we apply \eqref{cota a} with $m-1$ and Lemma~\ref{lemma2} (using \eqref{polin varios}) to get
\begin{align*}
 \sum_{\substack{\alpha \in \mathbb{N}_{0}^{n} \\ \vert \alpha \vert =m}} \vert c_{\alpha}(f) \vert r^{\alpha}
& = \sum_{1 \leq i_{1} \leq \ldots \leq i_{m} \leq n} \vert c_{ i_{1}  \ldots  i_{m}} \vert r_{i_{1}} \cdots r_{i_{m}} \\
& = \sum_{j =1}^{n} r_{j} \sum_{1 \leq i_{1} \leq \ldots \leq i_{m-1} \leq j} \vert c_{ i_{1}  \ldots  i_{m-1} j } \vert r_{i_{1}} \cdots r_{i_{m-1}} \\
& \leq \sum_{j =1}^{n} r_{j} \Big( \sum_{1 \leq i_{1} \leq \ldots \leq i_{m-1} \leq j} \vert c_{ i_{1}  \ldots  i_{m-1} j } \vert^{2} \Big)^{\frac{1}{2}}
 \Big( \sum_{1 \leq i_{1} \leq \ldots \leq i_{m-1} \leq j} r_{i_{1}}^{2} \cdots r_{i_{m-1}}^{2} \Big)^{\frac{1}{2}} \\
&  \leq \sup_{1 \leq j \leq n} r_{j} \Big( \sum_{1 \leq i_{1} \leq \ldots \leq i_{m-1} \leq j} r_{i_{1}}^{2} \cdots r_{i_{m-1}}^{2} \Big)^{\frac{1}{2}}
\sum_{j =1}^{n} \Big( \sum_{1 \leq i_{1} \leq \ldots \leq i_{m-1} \leq j} \vert c_{ i_{1}  \ldots  i_{m-1} j } \vert^{2} \Big)^{\frac{1}{2}} \\
& = \sup_{1 \leq j \leq n} r_{j} \Big( \sum_{\substack{\alpha \in \mathbb{N}_{0}^{j} \\ \vert \alpha \vert =m-1}} r^{2 \alpha}  \Big)^{\frac{1}{2}}
\sum_{j =1}^{n} \Big( \sum_{1 \leq i_{1} \leq \ldots \leq i_{m-1} \leq j} \vert c_{ i_{1}  \ldots  i_{m-1} j } \vert^{2} \Big)^{\frac{1}{2}} \\
& \leq \frac{\sqrt{C}}{a_{0}^{m-1}}  m (\sqrt{2})^{m-1} \big( 1 - \frac{1}{m-1} \big)^{m-1} \Vert P_{m,n} \Vert
\leq \tilde{C} \Big( \frac{\sqrt{2}}{a_{0}} \Big)^{m-1} m \Vert f \Vert \,.
\end{align*}
Now
\[
 \sum_{\alpha \in \mathbb{N}_{0}^{n}}  \vert c_{\alpha} (f) \vert r^{\alpha}
= \sum_{m=0}^{\infty} \sum_{\substack{\alpha \in \mathbb{N}_{0}^{n} \\ \vert \alpha \vert =m}} \vert c_{\alpha}(f) \vert r^{\alpha}
\leq \vert c_{0} \vert + \sum_{m=1}^{\infty} \tilde{C} \Big( \frac{\sqrt{2}}{a_{0}} \Big)^{m-1} m \Vert f \Vert \leq K \Vert f \Vert \,.
\]
Since this holds for every $n$ we finally have
\[
 \sum_{\alpha \in \mathbb{N}_{0}^{(\mathbb{N})}}  \vert c_{\alpha}(f) \vert r^{\alpha}
\leq K \Vert f \Vert \, .
\]
This implies $r \in \mon H_{\infty}(\mathbb{D}^{\infty})$ and completes the proof. \qed

\subsection{Dismissing  candidates} \label{sec:comments}

Our aim was to describe $\mon H_{\infty} (\mathbb{D}^{\infty})$ in terms of a sequence space; in other words, to find a sequence space $X$ such that
$X \cap \mathbb{D}^{\infty} = \mon H_{\infty} (\mathbb{D}^{\infty})$. The first natural candidate to do that job was $\ell_{2}$ but, as we already said in the introduction,
\cite[Theorem~1.1(b)]{BaCaQu06} implies that the sequence $(p_{n}^{-1/2})_{n}$ belongs to the set $\mon H_{\infty} (\mathbb{D}^{\infty})$ although it is not in $\ell_{2}$.
Then three other natural candidates are the spaces $\ell_{2,0}$ (already defined in Remark~\ref{ele20}), the Lorentz space $\ell_{2, \infty}$ and the space defined by
\[
 \ell_{2, \log} = \Big\{ z \in \ell_{\infty} \colon \exists c \, \forall n \, ;  z^{*}_{n} \leq c \sqrt{\textstyle\frac{\log n}{n}} \, , \, \,\forall n \Big\}  \, .
\]
Theorem~\ref{Leonhard} shows that neither $\ell_{2,0}$ nor $\ell_{2, \log}$ are the proper spaces, since we have
\begin{equation} \label{candidatos}
  \ell_{2,0}\cap\mathbb{D}^{\infty} \subsetneqq \frac{1}{\sqrt{2}} \mathbf{B} \subseteq \mon H_{\infty} (\mathbb{D}^{\infty})
\subseteq \mathbf{\bar{B}} \subsetneqq \ell_{2, \log}\cap \mathbb{D}^{\infty} \, .
\end{equation}
Before we proceed, let us note that since $\limsup_{n} \frac{1}{\log n} \sum_{j=1}^{n} \frac{1}{j} = 1$ we have that
\begin{align}
& \Big( \frac{c}{\sqrt{n}} \Big)_{n \in \mathbb{N}} \in \mon  H_{\infty} (\mathbb{D}^{\infty})  \text{ for all } c < \frac{1}{\sqrt{2}} \text{ and }  \label{raiz1} \\
& \Big( \frac{c}{\sqrt{n}} \Big)_{n \geq c^{2}} \not\in \mon  H_{\infty} (\mathbb{D}^{\infty})  \text{ for all } c > 1 \, .  \label{raiz2}
\end{align}
Now, \eqref{raiz1} immediately gives $\ell_{2,0}\cap\mathbb{D}^{\infty} \subsetneqq \frac{1}{\sqrt{2}} \mathbf{B}$.
Let us point out that this shows that Theorem~\ref{Leonhard} implies Remark~\ref{ele20}, giving more information.\\
The last inclusion in \eqref{candidatos} follows from \eqref{sup log}.
On the other hand,
\[
 \limsup_{n \to \infty} \frac{1}{\log n} \sum_{j=1}^{n} \Big( \sqrt{\frac{\log j}{j}} \Big)^{2}
\geq \limsup_{n \to \infty} \frac{1}{\log n} \sum_{j=1}^{n} \frac{\log 3}{j}
= \log 3 > 1
\]
gives $\Big( \sqrt{\frac{\log n}{n}} \Big)_{n} \not\in \mathbf{\bar{B}}$ and shows that this inclusion is also strict.\\
In certain steps of the proof of Theorem~\ref{Leonhard}-\textit{(\ref{Leo 1})} we have used the fact that, being $(r_{n})_{n}$ decreasing, $n r_{n}^{2} \leq \sum_{j=1}^{n}  r_{j}^{2}$.
One may wonder if then the condition $\limsup_{n} n z_{n}^{\ast 2} < 1/2$ is enough, but the previous comment shows that this is not the case.\\

Again, $\ell_{2, \infty}$ is not the right candidate; the situation now is slightly more complicated. First of all by \eqref{raiz1} we have
\[
\frac{1}{\sqrt{2}}B_{\ell_{2,\infty}}\subseteq \mon H_{\infty} (\mathbb{D}^{\infty}) \, .
 \]
Again, Theorem~\ref{Leonhard} is stronger than Remark~\ref{ele20}, since it
gives a concrete value to the $r$ obtained there. But more can be said;
\eqref{raiz2} shows that there are sequences  in $\ell_{2,\infty}\cap \mathbb{D}^\infty$ that do not belong to  $\mon H_{\infty} (\mathbb{D}^{\infty})$.
On the other hand there are sequences in $\mon H_{\infty} (\mathbb{D}^{\infty})$ that do not belong to $\ell_{2,\infty}$.  To check this claim, we take an strictly increasing
sequence of non-negative integers $(n_{k})_{k}$ with $n_{1}>1$,  satisfying that the sequence $(\frac{k+1}{n_{k}})_{k}$ is strictly decreasing and
\[
\sum_{k=1}^\infty \frac{k+1}{n_k}<\frac{1}{2} \, ;
\]
(take for example $n_k=a^{k^2(k+1)}$ for $a  \in \mathbb{N}$ big enough). Now we define
\[
 r_{j} =
\begin{cases}
 \sqrt{\frac{1}{n_{1}}} & 1\leq j \leq n_{1} \\
 \sqrt{\frac{k+1}{n_{k+1}}} & n_k <j\leq n_{k+1}, \ \ k=1,2,\ldots.
\end{cases}
\]
The sequence $(r_n)$ is  decreasing to 0. Clearly, $n_kr_{n_k}^2=k$ for all $k$. Thus $(r_n)$ does not belong to $\ell_{2,\infty}$. But  for $n>n_{1}$, if $n_k<n\leq n_{k+1}$, then
\begin{multline*}
 \frac{1}{\log n} \sum_{j=1}^{n} r_{j}^{2}=
 \frac{1}{\log n}\big( \sum_{j=1}^{n_{1}} \frac{1}{n_{1}}+\sum_{h=1}^{k-1}\sum_{j=n_h+1}^{n_{h+1}}r_j^2+
\sum_{j=n_k+1}^{n}r_j^2\big)\\
\leq  \frac{1}{\log n}\big( 1+\sum_{h=1}^{k-1}\frac{n_{h+1}-n_h}{n_{h+1}}(h+1)+\frac{n_{k+1}-n_k}{n_{k+1}}(k+1)\big)\\
\leq \frac{1}{\log n_{1}}+\sum_{h=1}^{k-1}\frac{h+1}{\log n_{h+1}}+\frac{k+1}{\log n_k}
< \sum_{h=1}^\infty \frac{h+1}{n_h}+\frac{k+1}{\log n_k}.
\end{multline*}
Hence $\limsup_{n} \frac{1}{\log n} \sum_{j=1}^{n} r_j^2 < \frac{1}{2}$, and therefore $(r_{n})_{n} \in \mon H_{\infty} (\mathbb{D}^{\infty})$.\\

In fact, Theorem~\ref{Leonhard}, through an argument like in \eqref{raiz1} and \eqref{raiz2}, shows that there is no Banach sequence space $X$ such that
$\mon H_{\infty} (\mathbb{D}^{\infty}) = X \cap \mathbb{D}^{\infty}$.

\section{Hardy spaces}
We draw now our attention to functions on $\mathbb{T}^{\infty}$, the infinite dimensional polytorus. We recall that $m$ denotes the product of the normalized Lebesgue measure
on $\mathbb{T}$. Given a function   $f  \in L_{p} (\mathbb{T}^{\infty})$, its Fourier coefficients $(\hat{f}(\alpha))_{\alpha \in \mathbb{Z}^{(\mathbb{N})}}$ are defined by
$\hat{f}(\alpha) = \int_{\mathbb{T}^{\infty}} f(w) w^{- \alpha} dm(w)= \langle f, w^{ \alpha} \rangle$ where $w^{\alpha}=w_{1}^{\alpha_{1}}\ldots w_n^{\alpha_n}$ if $\alpha=(\alpha_{1}\ldots \alpha_n, 0,\ldots)$ for $w \in \mathbb{T}^{\infty}$, and  the bracket $\langle \cdot , \cdot \rangle$ refers to the duality between $L_{p} (\mathbb{T}^{\infty})$ and $L_{q} (\mathbb{T}^{\infty})$ for $1/p + 1/q=1$.
With this, the Hardy spaces are defined as
\[
H_{p}(\mathbb{T}^{\infty}) = \Big\{ f \in  L_{p} (\mathbb{T}^{\infty}) \colon \hat{f}(\alpha) = 0 \,  , \, \,\, \,
\forall \alpha \in \mathbb{Z}^{(\mathbb{N})} \setminus \mathbb{N}_{0}^{(\mathbb{N})} \Big\}\, .
\]
We define the set
\begin{equation} \label{mon Hardy}
 \mon H_{p}(\mathbb{T}^{\infty}) = \Big\{ z \in \mathbb{D}^{\infty} \colon \sum_{\alpha} \vert \hat{f}(\alpha) z^{\alpha} \vert < \infty
\text{ for all }  f \in  H_{p}(\mathbb{T}^{\infty}) \Big\} \,.
\end{equation}
Our aim in this section is to determine these sets. As we did for holomorphic functions, we approach first the $m$-homogeneous case.

\subsection{The homogeneous case}

We consider, for each $m$, the following closed subspace
\begin{equation} \label{Hpm}
 H_{p}^{m}(\mathbb{T}^{\infty}) = \Big\{ f \in H_{p} (\mathbb{T}^{\infty}) \colon \hat{f} (\alpha) \neq 0 \, \Rightarrow \, \vert \alpha \vert = m \Big\} \,.
\end{equation}
of $L_{p} (\mathbb{T}^{\infty})$. By \cite[Section~9]{CoGa86} this is the completion of the $m$-homogeneous trigonometric polynomials (functions on
$\mathbb{T}^{\infty}$ that are finite sums of the form $\sum_{\vert \alpha \vert =m} c_{\alpha} w^{\alpha}$).  A sort of Khintchine-type inequality from \cite[9.1 Theorem]{CoGa86} shows
that, for $1 \leq p < \infty$, all  $\Vert \cdot \Vert_{p}$-norms are equivalent for the $m$-homogeneous trigonometric polynomials and then $H_{q}^{m}(\mathbb{T}^{\infty}) =
 H_{p}^{m}(\mathbb{T}^{\infty})$ for all $1 \leq p,q < \infty$ and all $m$ with equivalent norms.\\
Our aim now is to describe the set $\mon H_{p}^{m}(\mathbb{T}^{\infty})$, defined analogously to \eqref{mon Hardy}. We deal with
two sepparate situations: $p=\infty$ and $1 \leq p < \infty$. The first case will follow from Theorem~\ref{polinomios}, after showing that
$H_{\infty}^{m}(\mathbb{T}^{\infty})$ can be identified with $\mathcal{P}(^{m} c_{0})$. The basic idea here is, given a polynomial on $c_{0}$, extend
it to $\ell_{\infty}$ and then restrict it to $\mathbb{T}^{\infty}$. Let us very briefly recall how $m$-homogeneous polynomials on a Banach space $X$ can be
extended to its bidual (see \cite[Section~6]{Fl99} or  \cite[Proposition~1.53]{Di99}). First of all, every $m$-linear mapping $A: X \times \cdots \times X \to \mathbb{C}$
a unique extension (called the Arens extension) $\tilde{A} : X^{**} \times \cdots \times X^{**} \to \mathbb{C}$ such that for all $j=1, \ldots , n$, all
$x_{k} \in X$ and $z_{k} \in X^{**}$, the mapping that to $z \in X^{**}$ associates $\tilde{A}(x_{1}, \ldots , x_{j-1}, z , z_{j+1}, \ldots , z_{m})$ is weak$^{*}$-continuous.
Now, given $P \in \mathcal{P}(^{m} X)$, we take its associated symmetric $m$-linear form $A$ and define its Aron--Berner extension $\tilde{P} \in  \mathcal{P}(^{m} X^{**})$
by $\tilde{P}(z) = \tilde{A}(z, \ldots , z)$. By \cite[Theorem~3]{DaGa89} we have
\begin{equation} \label{DavieGamelin}
  \sup_{x \in B_{X}} \vert P(x) \vert = \sup_{z \in B_{X^{**}}} \vert P(z) \vert \, .
\end{equation}
Hence, the operator $AB:\mathcal{P}(^{m} X) \to \mathcal{P}(^{m} X^{**})$ defined by $AB(P) = \tilde{P}$ is a linear isometry.

\begin{proposition}\label{polco}
The mapping $\psi : \mathcal{P}(^{m} c_{0})  \to H_{\infty}^{m}(\mathbb{T}^{\infty})$ given by $$\psi (P)(w)= AB(P)(w)$$ for $w\in \mathbb{T}^{\infty}$ is
and $P\in \mathcal{P}(^{m} c_{0})$ a surjective isometry.
\end{proposition}
\begin{proof}
Let us note first that, by the very definition of the Aron--Berner extension, for each $\alpha \in \mathbb{N}_{0}^{(\mathbb{N})}$, the monomial
$x \in c_{0} \mapsto x^{\alpha}$ is extended to the monomial $z \in \ell_{\infty} \mapsto z^{\alpha}$. Then the set of finite sums of the type
$\sum_{\vert \alpha \vert} c_{\alpha} x^{\alpha}$ is bijectively and isometrically mapped onto the set of $m$-homogeneous trigonometric polynomials. By
\cite[Propositions~1.59 and 2.8]{Di99} the monomials on $c_{0}$ with $\vert \alpha \vert=m$ generate a dense subspace  of $\mathcal{P}(^m c_{0})$. On the other hand,
by \cite[Section~9]{CoGa86} the trigonometric polynomials are dense in $H_{\infty}^{m} (\mathbb{T}^{\infty})$. This gives the result.
\end{proof}

\noindent To deal with the case $1 \leq p < \infty$ we need the following lemma.

\begin{lemma} \label{inclusion}
 $\mon H_{p}^{m} (\mathbb{T}^{\infty}) \subseteq \mon H_{p}^{m-1} (\mathbb{T}^{\infty})$.
\end{lemma}
\begin{proof}
 Let $0 \neq z \in \mon H_{p}^{m} (\mathbb{T}^{\infty})$ and $f \in \mon H_{p}^{m-1} (\mathbb{T}^{\infty})$. We choose
$z_{i_{0}} \neq 0$ and define $\tilde{f} (w) = w_{i_{0}} f(w)$. Let us see that $\tilde{f} \in H_{p}^{m} (\mathbb{T}^{\infty})$; indeed, take a sequence
$(f_{n})_{n}$ of $(m-1)$-homogeneous trigonometric polynomials that converges in the space $L_{p}(\mathbb{T}^{\infty})$ to $f$. Each $f_{n}$ is a finite sum of the type
$\sum_{\vert \alpha \vert = m-1} c_{\alpha}^{(n)} w^{\alpha}$. We define for $w \in \mathbb{T}^{\infty}$
\[
 \tilde{f}_{n} (w) = w_{i_{0}} f_{n}(w) = \sum_{\vert \alpha \vert = m-1} c_{\alpha}^{(n)} w_{1}^{\alpha_{1}} \cdots w_{i_{0}}^{\alpha_{i_{0}}+1} \cdots w_{k}^{\alpha_{k}} \, .
\]
Clearly $\tilde{f}_{n}$ is an $m$-homogeneous trigonometric polynomial. Moreover
\begin{multline*}
 \Big( \int_{\mathbb{T}^{\infty}} \vert w_{i_{0}} f_{n}(w) - w_{i_{0}} f(w) \vert^{p}  d m(w) \Big)^{\frac{1}{p}}
= \Big( \int_{\mathbb{T}^{\infty}} \vert w_{i_{0}}\vert^{p}  \vert f_{n}(w) - f(w) \vert^{p}  dm(w) \Big)^{\frac{1}{p}}  \\
\leq \Big( \int_{\mathbb{T}^{\infty}}   \vert f_{n}(w) - f(w) \vert^{p}  d m(w) \Big)^{\frac{1}{p}} \, .
\end{multline*}
The last term converges to $0$, hence $(\tilde{f}_{n})_{n}$ converges in $L_{p}(\mathbb{T}^{\infty})$ to$\tilde{f}$ and $\tilde{f} \in H_{p}^{m} (\mathbb{T}^{\infty})$. We compute
now the Fourier coefficients:
\begin{multline*}
 \hat{\tilde{f}} (\alpha) = \int_{\mathbb{T}^{\infty}} \tilde{f}(w) w^{- \alpha} dm(w)
= \int_{\mathbb{T}^{\infty}} w_{i_{0}} f(w) w^{- \alpha} dm(w) \\
= \int_{\mathbb{T}^{\infty}}  f(w) w_{1}^{-\alpha_{1}} \cdots w_{i_{0}}^{-\alpha_{i_{0}}+1} \cdots w_{n}^{-\alpha_{n}} dm(w)  \\
= \int_{\mathbb{T}^{\infty}}  f(w) w_{1}^{-\alpha_{1}} \cdots w_{i_{0}}^{-(\alpha_{i_{0}}-1)} \cdots w_{n}^{-\alpha_{n}} dm(w)     \\
= \hat{f} (\alpha_{1} , \ldots , \alpha_{i_{0}}-1 , \ldots , \alpha_{n}) \, .
\end{multline*}
That is
\[
  \hat{\tilde{f}} (\alpha) =
\begin{cases}
\hat{f}(\beta) & \text{ if } \alpha = (\beta_{1} , \ldots , \beta_{i_{0}} + 1 , \ldots , \beta_{n}) \\
0  & \text{ otherwise }
\end{cases}
\]
and this gives
\begin{multline*}
 \sum_{\beta} \vert \hat{f} (\beta) z^{\beta} \vert = \frac{1}{\vert z_{i_{0}} \vert} \sum_{\beta} \vert \hat{f} (\beta) z^{\beta} \vert \, \vert z_{i_{0}} \vert \\
= \frac{1}{\vert z_{i_{0}} \vert} \sum_{\beta} \vert \hat{f} (\beta) z_{1}^{\beta_{1}} \cdots z_{i_{0}}^{\beta_{i_{0}}+1} \cdots z_{n}^{\beta_{n}}  \vert
= \sum_{\alpha} \vert \hat{\tilde{f}} (\alpha) z^{\alpha} \vert < \infty \, .
\end{multline*}
Hence $z \in \mon H_{p}^{m-1} (\mathbb{T}^{\infty})$.
\end{proof}

To give the description of $\mon H_{p}^{m}(\mathbb{T}^{\infty})$ we aim at we are going to use the following result, that re-proves the know fact from \cite{CoGa86} that
$H_{1}^{m}(\mathbb{T}^{\infty}) = H_{2}^{m}(\mathbb{T}^{\infty})$ but with a more precise control of the constants on
the equivalence of the norms. We get this control from \eqref{eq:bonami 1}, the inequality on polynomials on finitely many variables that we already used in the proof of
Lemma~\ref{lemma2}.

\begin{lemma} \label{bonami}
We have $H_{1}^{m}(\mathbb{T}^{\infty}) = H_{2}^{m}(\mathbb{T}^{\infty})$ and for all $f$
\begin{equation} \label{eq:bonami}
\Vert f \Vert_{1} \leq \Vert f \Vert_{2} \leq (\sqrt{2})^{m} \Vert f \Vert_{1} \, .
\end{equation}
\end{lemma}
\begin{proof}
If $f$ is a trigonometric polynomial, then there is a finite set $J$ of multi-indices of order $m$ such that $f(w)=\sum_{\alpha \in J} c_{\alpha} w^{\alpha}$ for all
$w \in \mathbb{T}^{\infty}$. But, being $J$ finite, there is $k$ such that $J \subseteq \Lambda (m,k)$ and $\sum_{\alpha \in J} c_{\alpha} z^{\alpha}$, now
for $z \in \mathbb{C}^{k}$, defines an $m$-homogeneous polynomial in $k$ variables. Then \eqref{eq:bonami 1} gives that \eqref{eq:bonami} holds for every
trigonometric polynomial. Since these are dense both in $H_{1}^{m}(\mathbb{T}^{\infty})$ and in $H_{2}^{m}(\mathbb{T}^{\infty})$ we have that both spaces are equal
and \eqref{eq:bonami} holds for every $f$.
\end{proof}

\begin{theorem} \label{final}
\[
  \mon H_{p}^{m} (\mathbb{T}^{\infty}) =
\begin{cases}
  \ell_{2} \cap \mathbb{D}^{\infty} & \text{ for } 1 \leq p < \infty \\
  \ell_{\frac{2m}{m-1}, \infty} \cap \mathbb{D}^{\infty} & \text{ for }  p = \infty
\end{cases}
\]
Moreover, there exists some universal constant $C>0$ such that if $z \in \mon H_{p}^{m} (\mathbb{T}^{\infty})$ and  $f \in H_{p}^{m} (\mathbb{T}^{\infty})$
then
\begin{equation} \label{normas final}
\sum_{\vert \alpha \vert =m} \vert \hat{f} (\alpha) z^{\alpha} \vert \leq C^{m} \Vert z \Vert^{m} \Vert f \Vert_{p} \, ,
\end{equation}
where $\Vert z \Vert$ is the norm in the corresponding sequence space.
\end{theorem}
\begin{proof}
The case $p=\infty$ follows from Theorem~\ref{polinomios} and Proposition~\ref{polco}. For $p < \infty$ let us first observe that by \cite[9.1~Theorem]{CoGa86}
$H_{p}^{m} (\mathbb{T}^{\infty}) = H_{2}^{m} (\mathbb{T}^{\infty})$ with equivalent norms, hence
\[
\mon   H_{p}^{m} (\mathbb{T}^{\infty}) = \mon H_{2}^{m} (\mathbb{T}^{\infty})
\]
for every $1 \leq p < \infty$ and all $m$ and it suffices to handle the case $p=2$. If $z \in \ell_{2} \cap \mathbb{D}^{\infty}$ and $f \in H_{2}^{m}(\mathbb{T}^{\infty})$
we apply the Cauchy-Schwarz inequality and the binomial formula to get
\begin{equation} \label{CauchySchwarz}
\sum_{\substack{\alpha \in \mathbb{N}_{0}^{(\mathbb{N})} \\ \vert \alpha \vert = m}} \vert \hat{f} (\alpha) z^{\alpha} \vert
\leq \Big( \sum_{\substack{\alpha \in \mathbb{N}_{0}^{(\mathbb{N})} \\ \vert \alpha \vert = m}} \vert \hat{f}(\alpha) \vert^{2} \Big)^{\frac{1}{2}}
\Big( \sum_{\substack{\alpha \in \mathbb{N}_{0}^{(\mathbb{N})} \\ \vert \alpha \vert = m}} \vert z \vert^{2 \alpha} \Big)^{\frac{1}{2}}
\leq \Vert z \Vert_{2}^{m} \Vert f \Vert_{2} < \infty \, ;
\end{equation}
this implies $z \in \mon H_{2}^{m}(\mathbb{T}^{\infty})$. \\
Let us now fix $z \in \mon H_{2}^{1}(\mathbb{T}^{\infty})$. By a closed-graph argument, there exists $c_{z}$ such that for every $f \in H_{2}^{1}(\mathbb{T}^{\infty})$
the inequality
$\sum_{n=1}^{\infty} \vert \hat{f}(n) z_{n} \vert \leq c_{z} \Vert f \Vert_{2}$ holds.
We fix now $y \in \ell_{2}$ and for each $N \in \mathbb{N}$ we define a function $f_{N} : \mathbb{T}^{\infty} \to \mathbb{C}$ by
$f_{N}(w)= \sum_{n=1}^{N} w_{n} y_{n}$. Clearly $\hat{f}(n) = y_{n}$ for $n=1, \ldots , N$ and $f \in H_{2}^{1}(\mathbb{T}^{\infty})$; then
we have
\[
\sum_{n=1}^{N} \vert \hat{f}(n) z_{n} \vert \leq c_{z} \Big(  \sum_{n=1}^{N} \vert y_{n} \vert^{2} \Big)^{\frac{1}{2}}
\leq  c_{z} \Vert y \Vert_{2} < \infty \, .
\]
This holds for every $N$, hence $\sum_{n=1}^{\infty} \vert \hat{f}(n) z_{n} \vert \leq  c_{z} \Vert y \Vert_{2}$ and, since this holds for every $y \in \ell_{2}$,
we have $z \in \ell_{2}$. This gives
\[
 \ell_{2} \cap \mathbb{D}^{\infty} \subseteq \mon H_{2}^{m} (\mathbb{T}^{\infty})
\subseteq \mon H_{2}^{1} (\mathbb{T}^{\infty}) \subseteq \ell_{2} \cap \mathbb{D}^{\infty} \, .
\]
Finally, inequality \eqref{normas final} follows immediately from Theorem~\ref{polinomios} and Proposition~\ref{polco} for the case $p=\infty$. Inequality \eqref{CauchySchwarz}
gives  \eqref{normas final} for $2 \leq p < \infty$ with $C \leq 1$. Finally \eqref{CauchySchwarz} and \eqref{eq:bonami} give the inequality with
$C \leq \sqrt{2}$ whenever $1 \leq p < 2$.
\end{proof}

\subsection{The general case}

We address now our main goal of describing $\mon H_{p} (\mathbb{T}^{\infty})$. As it often happens (and was also the case in the $m$-homogeneous setting), there
are only two significant cases: $p= \infty$ and  $p=2$. The description of $\mon H_{p}(\mathbb{T}^{\infty})$  for $1 \leq p < \infty$
will follow from the cases $p=2$ and $p=1$, showing that these two coincide. We state our main result.
\begin{theorem} \label{general}$\ $
\begin{enumerate}
  \item \label{general inf} We have $\mon H_{\infty} (\mathbb{T}^{\infty}) = \mon H_{\infty} (\mathbb{D}^{\infty})$. In particular, by Theorem~\ref{Leonhard}
\[
  \frac{1}{\sqrt{2}} \mathbf{B} \,\subseteq \,\mon H_{\infty}(\mathbb{T}^{\infty}) \,\subseteq\, \mathbf{\overline{B}}\, .
\]

  \item \label{general p} For $1 \leq p < \infty$ we have  $\mon H_{p}(\mathbb{T}^{\infty}) = \ell_{2} \cap  \mathbb{D}^{\infty}$.
\end{enumerate}
\end{theorem}

Part \textit{(\ref{general inf})} follows immediately from \eqref{00} and the following result, that is known (see \cite[11.2 Theorem]{CoGa86}
and \cite[Lemma~2.3]{HeLiSe97}); we include an elementary direct proof of it for the sake of completeness. The statement about the inverse mapping seems to be new.
\begin{proposition} \label{H intftys}
 There exists a unique surjective isometry $\phi : H_{\infty} (\mathbb{T}^{\infty}) \rightarrow H_{\infty} (\mathbb{D}^{\infty}_{0})$ such that
$c_{\alpha} \big(\phi (f)\big) = \hat{f}(\alpha)$ for every $f \in H_{\infty} (\mathbb{T}^{\infty})$ and every $\alpha \in \mathbb{N}_{0}^{(\mathbb{N})}$.\\
Moreover, when restricted to $H_{\infty}^{m} (\mathbb{T}^{\infty})$, the mapping $\psi$ defined in Proposition~\ref{polco} and $\phi$ are inverse to each other.
\end{proposition}
\begin{proof}
First of all, let us note that in the finite dimensional setting the result is true: It is a well known fact (see e.g. \cite[3.4.4 exercise (c)]{Ru69}) that for each $n$ there exists
an isometric bijection $\phi_{n} : H_{\infty} (\mathbb{T}^{n}) \rightarrow H_{\infty} (\mathbb{D}^{n})$ such that
$c_{\alpha} \big(\phi (f)\big) = \tilde{f}(\alpha)$ for every $f \in H_{\infty} (\mathbb{T}^{n})$ and every $\alpha \in \mathbb{N}_{0}^{n}$.\\
Take now $f \in H_{\infty} (\mathbb{T}^{\infty})$ and fix $n \in \mathbb{N}$; since we can consider  $\mathbb{T}^{\infty} = \mathbb{T}^{n} \times \mathbb{T}^{\infty}$, we write $w  = (w_{1}, \ldots , w_{n} , \tilde{w}_{n})$ $\in \mathbb{T}^{\infty}$. Then we define
$f_{n} : \mathbb{T}^{n}  \rightarrow \mathbb{C}$ by
\[
 f_{n}(w_{1} , \ldots w_{n}) = \int_{\mathbb{T}^{\infty}} f (w_{1}, \ldots , w_{n} , \tilde{w}_{n}) dm(\tilde{w}_{n}) \, .
\]
By the Fubini theorem $f_{n}$ is well defined a.e. and
\[
 \int_{\mathbb{T}^{\infty}} f (w) d m(w)
=   \int_{\mathbb{T}^{n}} \bigg( \int_{\mathbb{T}^{\infty}} f (w_{1}, \ldots , w_{n} , \tilde{w}_{n}) d m (\tilde{w}_{n}) \bigg) d m_{n} (w_{1}, \ldots , w_{n}) \, ,
\]
hence $f_{n} \in L_{\infty} (\mathbb{T}^{n})$. Moreover, for $\alpha \in \mathbb{Z}^{n}$ we have, again by Fubini
\[
 \hat{f}_{n} (\alpha) =
\int_{\mathbb{T}^{n} \times \mathbb{T}^{\infty}} f(w) w^{- \alpha} dm(w) = \hat{f} (\alpha) \, .
\]
Thus $\hat{f}_{n} (\alpha) = \hat{f} (\alpha) = 0$ for every $\alpha \in \mathbb{Z}^{n} \setminus \mathbb{N}^{n}_{0}$ and $f_{n} \in H_{\infty} (\mathbb{T}^{n})$.
Obviously $\Vert f_{n} \Vert_{\infty} \leq \Vert f \Vert_{\infty}$ since the measure is a probability. We take $g_{n} = \phi_{n} (f_{n}) \in H_{\infty}(\mathbb{D}^{n})$. We have
$\Vert g_{n} \Vert_{\infty} = \Vert f_{n} \Vert_{\infty} \leq \Vert f \Vert_{\infty}$ and
\[
 g_{n} (z) = \sum_{\alpha \in \mathbb{N}_{0}^{n}} \hat{f}_{n} (\alpha) z^{\alpha} = \sum_{\alpha \in \mathbb{N}_{0}^{n}} \hat{f} (\alpha) z^{\alpha}
\]
for every $z \in \mathbb{D}^{n}$. Since this holds for every $N$ we can define $g : \mathbb{D}^{(\mathbb{N})} \rightarrow \mathbb{C}$ by
$g(z) = \sum_{\alpha \in \mathbb{N}_{0}^{n}} \hat{f} (\alpha) z^{\alpha}$. We have $\Vert g \Vert_{\infty} = \sup_{n} \Vert g_{n} \Vert_{\infty} \leq \Vert f \Vert_{\infty}$.
 By \cite[Lemma~2.2]{DeGaMa10} there exists a unique extension $\tilde{g} \in H_{\infty} (\mathbb{D}_{0}^{\infty})$ with $c_{\alpha} (\tilde{g}) =\hat{f} (\alpha)$
and $\Vert \tilde{g} \Vert_{\infty} = \Vert g \Vert_{\infty}  \leq \Vert f \Vert_{\infty}$. Setting $\phi(f) = \tilde{g}$ we have that $\phi : H_{\infty} (\mathbb{T}^{\infty})
\rightarrow H_{\infty} (\mathbb{D}_{0}^{\infty})$ is well defined and such that  for every $f \in H_{\infty} (\mathbb{T}^{\infty})$ we have $\Vert \phi(f) \Vert_{\infty}
\leq \Vert f \Vert_{\infty}$ and $c_{\alpha} \big(\phi (f)\big) = \tilde{f}(\alpha)$  for every $\alpha \in \mathbb{N}_{0}^{(\mathbb{N})}$.\\
On the other hand if $f \in L_{\infty}(\mathbb{T}^{\infty})$ is such that $\hat{f}(\alpha) =0$ for all $\alpha$ then $f =0$. Hence $\phi$ is injective.\\
Let us see that it is also surjective and moreover an isometry. Fix $g \in H_{\infty} (\mathbb{D}_{0}^{\infty})$ and consider $g_{n}$ its restriction to the first $n$ variables.
Clearly $g_{n} \in  H_{\infty} (\mathbb{D}^{n})$ and $\Vert g_{n} \Vert_{\infty} \leq \Vert g \Vert_{\infty}$. Using again \cite[3.4.4 exercise (c)]{Ru69} we can
choose $f_{n} \in   H_{\infty} (\mathbb{T}^{n})$ such that $\Vert f_{n} \Vert_{\infty} = \Vert g_{n} \Vert_{\infty}$ and $c_{\alpha} (g_{n}) = \hat{f}_{n} (\alpha)$
for all $\alpha \in \mathbb{N}_{0}^{n}$. Since $c_{\alpha} (g_{n}) = c_{\alpha} (g)$ we have $\hat{f}_{n} (\alpha)= c_{\alpha} (g)$.
We define now $\tilde{f}_{n} \in H_{\infty} (\mathbb{T}^{\infty})$ by $\tilde{f}_{n} (w) = f_{n} (w_{1} , \ldots , w_{n})$ for
$w \in \mathbb{T}^{\infty}$. Then the sequence $(\tilde{f}_{n})_{n=1}^{\infty}$ is contained in the closed ball in $L_{\infty} (\mathbb{T}^{\infty})$
centered at $0$ and with radius $\Vert g \Vert_{\infty}$. Since this ball is $w^{*}$\!-compact and metrizable,
there is a subsequence $(\tilde{f}_{n_{k}})_{k}$ that $w^{*}$-converges to some $f \in L_{\infty} (\mathbb{T}^{\infty})$
with $\Vert f \Vert_{\infty} \leq \Vert g \Vert_{\infty}$. Moreover,  $\hat{f}(\alpha)
= \langle f , w^{ \alpha} \rangle =$ $ \lim_{k \to \infty} \langle \tilde{f}_{n_{k}} , w^{ \alpha} \rangle = $ $ \lim_{k \to \infty} \hat{\tilde{f}}_{n_{k}} (\alpha)$ for every
$\alpha \in \mathbb{Z}_{0}^{(\mathbb{N})}$ and this implies $f \in H_{\infty} (\mathbb{T}^{\infty})$.
Let us see that  $\phi(f)=g$, which shows that $\phi$ is onto; indeed, if $\alpha = (\alpha_{1}, \ldots , \alpha_{n_{0}}, 0 , \ldots)$ then for $n_{k} \geq n_{0}$ we have
\[
\langle \tilde{f}_{n_{k}} , w^{\alpha} \rangle
= \int_{\mathbb{T}^{\infty}} \tilde{f}_{n_{k}} (w) w^{-\alpha} dm(w)
= \int_{\mathbb{T}^{n_{k}}} f_{n_{k}} (w) w^{-\alpha} dm_{n_{k}}(w)
= \hat{f}_{n_{k}} (\alpha) = c_{\alpha} (g) \, .
\]
Hence  $\hat{f}(\alpha)= c_{\alpha} (g)$ for all $\alpha \in \mathbb{N}_{0}^{(\mathbb{N})}$. Furthermore, since $\Vert f \Vert_{\infty} \leq \Vert g \Vert_{\infty} = \Vert \phi(f) \Vert_{\infty}$ we also get that
$\phi$ is an isometry.\\
Let us fix $P \in \mathcal{P}(^{m} c_{0})$ and show that $\phi^{-1}(P)(w)= \tilde{P}(w)$ for every $w \in \mathbb{T}^{\infty}$. We
choose $(J_{k})_{k}$ a sequence of finite families of multi-indexes included in $\{\alpha: \alpha\in \mathbb{N}_{0}^{(\mathbb{N})}: \vert \alpha \vert =m\}$ and such that the
sequence  $P_{k}=\sum_{\alpha\in J_k} c_{\alpha,k} x^{\alpha}$ converges uniformly to $P$ on the unit ball of $c_{0}$. Since each $J_k$ is finite, we have
\[
\phi^{-1}(P_k)(w)=\sum_{\alpha\in J_k}c_{\alpha,k}w^{\alpha}=\tilde{P}_k(w) \, ,
\]
for every $w\in \mathbb{T}^\infty$. The linearity of the $AB$ operator and \eqref{DavieGamelin} give that  $\Vert \tilde{P}-\tilde{P}_{k} \Vert= \Vert P-P_{k} \Vert =
\Vert \phi^{-1}(P)-\phi^{-1}(P_{k}) \Vert$ converges to 0 and complete the proof.\\
Observe that this argument actually works to prove that $\phi^{-1}(g)(w)=\tilde{g}(w)$ for every $w\in \mathbb{T}^\infty$ and every function $g$ in the completion of the space
 of all polynomials on $c_{0}$.
\end{proof}

\noindent For the proof of part \textit{(\ref{general p})}  of Theorem~\ref{general} we need some previous work. We handle first the case $p=2$. Here, since
$H_{2}(\mathbb{T}^{\infty})$ is a Hilbert space where $\{w^{\alpha} \}_{\alpha \in \mathbb{N}_{0}^{(\mathbb{N})}}$
is an orthonormal basis we have $\Vert f \Vert_{2} = \big( \sum_{\alpha} \vert \hat{f}(\alpha) \vert^{2} \big)^{\frac{1}{2}}$. This simplifies a lot the problem and we can
readily get the result in this case.
\begin{theorem} \label{H2}
  We have $\mon H_{2}(\mathbb{T}^{\infty}) = \ell_{2} \cap \mathbb{D}^{\infty}$ and for each $z \in \ell_{2} \cap \mathbb{D}^{\infty}$ and
$f \in H_{2}(\mathbb{T}^{\infty})$,
\begin{equation} \label{eq:H2}
  \sum_{\alpha \in \mathbb{N}_{0}^{(\mathbb{N})}} \vert \hat{f} (\alpha) z^{\alpha} \vert
\leq \Big( \prod_{n=1}^{\infty} \frac{1}{1-\vert z_{n} \vert^{2}} \Big)^{\frac{1}{2}} \Vert f \Vert_{2} \, .
\end{equation}
Moreover, the constant $\big( \prod_{n} \frac{1}{1-\vert z_{n} \vert^{2}} \big)^{1/2}$ is optimal.
\end{theorem}
\begin{proof}
The fact that $\ell_{2} \cap \mathbb{D}^{\infty} \subseteq \mon H_{2}(\mathbb{T}^{\infty})$ follows by using the Cauchy-Schwarz inequality in a similar way as un in
\eqref{CauchySchwarz}:
\[
  \sum_{\alpha \in \mathbb{N}_{0}^{(\mathbb{N})}} \vert \hat{f} (\alpha) z^{\alpha} \vert
\leq \Big( \sum_{\alpha \in \mathbb{N}_{0}^{(\mathbb{N})}} \vert \hat{f}(\alpha) \vert^{2} \Big)^{\frac{1}{2}}
\Big( \sum_{\alpha \in \mathbb{N}_{0}^{(\mathbb{N})}} \vert z \vert^{2 \alpha} \Big)^{\frac{1}{2}}
= \Vert f \Vert_{2} \Big( \prod_{n=1}^{\infty} \frac{1}{1-\vert z_{n} \vert^{2}} \Big)^{\frac{1}{2}} < \infty \, .
\]

\noindent On the other hand, since $H_{2}^{1}(\mathbb{T}^{\infty}) \subseteq H_{2}(\mathbb{T}^{\infty})$ we have that $\mon H_{2}(\mathbb{T}^{\infty})$ is a  subset of
$\mon H_{2}^{1}(\mathbb{T}^{\infty})$ and Theorem~\ref{final} gives the conclusion.\\
To see that the constant in the inequality is optimal, let us fix $z$ in $\mon H_{2}(\mathbb{T}^{\infty})$ and take $c>0$ such that
\[
   \sum_{\alpha \in \mathbb{N}_{0}^{(\mathbb{N})}} \vert \hat{f} (\alpha) z^{\alpha} \vert \leq c \Vert f \Vert_{2} \, .
\]
For each $n \in \mathbb{N}$ we consider the function $f_{n}(w)= \sum_{\alpha \in \mathbb{N}_{0}^{n}} z^{\alpha} w^{\alpha}$ that clearly satisfies
$f_{z} \in H_{2}(\mathbb{T}^{\infty})$ and $\hat{f}_{z}(\alpha) = z^{\alpha}$ for every $\alpha \in \mathbb{N}_{0}^{n}$ (and $0$ otherwise). Hence
\[
  \sum_{\alpha \in \mathbb{N}_{0}^{n}} \vert z^{\alpha} \vert^{2}
= \sum_{\alpha \in \mathbb{N}_{0}^{n}} \vert \hat{f}_{z} (\alpha) z^{\alpha} \vert
\leq c \Vert f \Vert_{2} = c \Big( \sum_{\alpha \in \mathbb{N}_{0}^{n}} \vert z^{\alpha}  \vert^{2} \Big)^{\frac{1}{2}} \, .
\]
This gives
\[
  c \geq  \Big( \sum_{\alpha \in \mathbb{N}_{0}^{n}} \vert z^{\alpha}  \vert^{2} \Big)^{\frac{1}{2}}
= \Big( \prod_{n=1}^{n} \frac{1}{1-\vert z_{n} \vert^{2}} \Big)^{\frac{1}{2}}
\]
for every $n$. Hence $ c \geq  \big( \prod_{n=1}^{\infty} \frac{1}{1-\vert z_{n} \vert^{2}} \big)^{1/2}$ and the proof is completed.
\end{proof}
Since $H_{p}(\mathbb{T}^{\infty}) \subseteq H_{q}(\mathbb{T}^{\infty})$ for all $p \geq q$, the previous result gives $\ell_{2} \cap \mathbb{D}^{\infty}
\subseteq \mon H_{p}(\mathbb{T}^{\infty})$ for every $2 \leq p$. To get the remaining case $1 \leq p < 2$ we need to somehow relate $H_{1}(\mathbb{T}^{\infty})$
and $H_{2}(\mathbb{T}^{\infty})$. \\
We are going to use now the fact that $H_{1}^{m}(\mathbb{T}^{\infty}) = H_{2}^{m}(\mathbb{T}^{\infty})$ with the control of the constants on
the equivalence of the norms that we got in Lemma~\ref{bonami}.\\

\noindent We will also need the following lemma, an $H_{p}$--version of \cite[Satz~VI]{Bo13_Goett} (see also \cite[Lemma~2]{DeGaMaPG08}).
\begin{lemma} \label{annalen}
 Let $z \in \mon H_{p}(\mathbb{T}^{\infty})$  and $x = (x_{n})_{n} \in \mathbb{D}^{\infty}$ such that $\vert x_{n} \vert \leq \vert z_{n} \vert$ for
all but finitely many $n$'s. Then $x \in  \mon H_{p}(\mathbb{T}^{\infty})$.
\end{lemma}
\begin{proof}
 We follow  \cite[Lemma~2]{DeGaMaPG08} and choose $r \in \mathbb{N}$ such that $\vert x_{n} \vert \leq \vert z_{n} \vert$ for all $n > r$.
We also take $a>1$ such that $\vert z_{n} \vert < \frac{1}{a}$ for $n=1 , \ldots , r$. Let $f \in H_{p}(\mathbb{T}^{\infty})$ with $\Vert f \Vert_{p} \leq 1$.
We fix $n_{1}, \ldots , n_{r} \in \mathbb{N}$ and define for each $u \in \mathbb{T}^{\infty}$,
\[
 f_{n_{1} , \ldots , n_{r}}(u) = \int_{\mathbb{T}^{r}} f(w_{1}, \ldots , w_{r}, u_{1}, \ldots ) w_{1}^{-n_{1}} \cdots  w_{r}^{-n_{r}} d m_{r} (w_{1} , \ldots ,w_{r}) \, .
\]
Let us see that $f_{n_{1} , \ldots , n_{r}} \in H_{p}(\mathbb{T}^{\infty})$; indeed, using H\"{o}lder  inequality we have
\begin{multline*}
 \bigg( \int_{\mathbb{T}^{\infty}}  \vert f_{n_{1} , \ldots , n_{r}}  (u) \vert^{p} dm(u) \bigg)^{\frac{1}{p}} \\
 = \bigg( \int_{\mathbb{T}^{\infty}} \Big\vert \int_{\mathbb{T}^{r}} f(w_{1}, \ldots , w_{r}, u_{1}, \ldots ) w_{1}^{-n_{1}} \cdots  w_{r}^{-n_{r}}
dm_{r} (w_{1} , \ldots , w_{r})  \Big\vert^{p} dm(u) \bigg)^{\frac{1}{p}} \\
\leq \bigg( \int_{\mathbb{T}^{\infty}} \Big( \int_{\mathbb{T}^{r}} \vert f(w_{1}, \ldots , w_{r}, u_{1}, \ldots )  \vert^{p}
dm_{r}(w_{1} , \ldots , w_{r})   dm(u) \bigg)^{\frac{1}{p}} = \Vert f  \Vert_{p} \, .
\end{multline*}
Hence $f_{n_{1} , \ldots , n_{r}} \in L_{p} (\mathbb{T}^{\infty})$ and $\Vert f_{n_{1} , \ldots , n_{r}} \Vert_{p} \leq  \Vert f \Vert_{p} \leq 1$.\\
Now we have, for $\alpha = (\alpha_{1}, \ldots , \alpha_{k}, 0, \ldots)$
\begin{align*}
 \hat{f}_{n_{1} , \ldots , n_{r}}(\alpha)
= & \int_{\mathbb{T}^{\infty}} f_{n_{1} , \ldots , n_{r}} (u) u^{- \alpha} dm(u) \\
& \bigg( \int_{\mathbb{T}^{\infty}} \int_{\mathbb{T}^{r}} \frac{f(w_{1}, \ldots , w_{r}, u_{1}, \ldots , u_{k} ) }%
{w_{1}^{n_{1}} \cdots  w_{r}^{n_{r}} u_{1}^{ \alpha_{1}} \cdots u_{k}^{ \alpha_{k}}} dm_{r} (w_{1}, \ldots , w_{r}) dm_{k} (u_{1}, \ldots , u_{k}) \bigg)  \\
= & \hat{f}(n_{1} , \ldots , n_{r}, \alpha_{1}, \ldots , \alpha_{k}) \, .
\end{align*}
Therefore
\[
  \hat{f}_{n_{1} , \ldots , n_{r}}(\alpha) =
\begin{cases}
 \hat{f}(n_{1} , \ldots , n_{r}, \alpha_{1}, \ldots , \alpha_{k})  & \text{ if } \alpha= (0,\stackrel{r}{\ldots} , 0, \alpha_{1}, \ldots ) \\
0 & \text{ otherwise}
\end{cases}
\]
and this implies  $f_{n_{1} , \ldots , n_{r}} \in H_{p} (\mathbb{T}^{\infty})$. Now, using \eqref{Baire}
and doing exactly the same calculations as in \cite[Lemma 2]{DeGaMaPG08}
we conclude $\sum_{\alpha} \vert \hat{f}(\alpha) x^{\alpha} \vert < \infty$ and  $x$ belongs to $\mon H_{p}(\mathbb{T}^{\infty})$.
\end{proof}

\begin{proof}[Proof of Theorem~\ref{general}--(\ref{general p})]
 Let us remark first that since $H_{p}(\mathbb{T}^{\infty}) \subseteq H_{1}(\mathbb{T}^{\infty})$ we have $\mon H_{1}(\mathbb{T}^{\infty}) \subseteq
\mon H_{p}(\mathbb{T}^{\infty})$. Then to get the lower bound it is enough to show
that $\ell_{2} \cap  \mathbb{D}^{\infty} \subseteq \mon H_{1}(\mathbb{T}^{\infty}) $. As a first step we show that
there exists $0<r< 1$ such that $r B_{\ell_{2}} \cap \mathbb{D}^{\infty} \subseteq \mon H_{1}(\mathbb{T}^{\infty})$.  Let $r < 1 / \sqrt{2}$ and choose $f \in H_{1}(\mathbb{T}^{\infty})$ and $z \in r B_{\ell_{2}} \cap \mathbb{D}^{\infty}$.  Then
$z = r y$ for some $y \in  B_{\ell_{2}}$. By \cite[9.2 Theorem]{CoGa86} there exists a projection
$P_{m} : H_{1}(\mathbb{T}^{\infty}) \to H_{1}^{m}(\mathbb{T}^{\infty})$ such that $\Vert P_{m} g \Vert_{1} \leq \Vert g \Vert_{1}$ for every
$g \in  H_{1}(\mathbb{T}^{\infty})$. We write $f_{m} = P_{m} (f)$ and we have $\hat{f}_{m}(\alpha) = \hat{f}(\alpha)$ if $\vert \alpha \vert = m$ and $0$ otherwise.
Then
\begin{multline*}
 \sum_{\alpha} \vert \hat{f}(\alpha) z^{\alpha} \vert = \sum_{m=0}^{\infty} \sum_{\vert \alpha \vert =m} \vert \hat{f}(\alpha) (r y)^{\alpha} \vert
= \sum_{m=0}^{\infty} \sum_{\vert \alpha \vert =m} \vert \hat{f}_{m}(\alpha) (r y)^{\alpha} \vert \\
\leq K \sum_{m=0}^{\infty} r^{m} \Vert f_{m} \Vert_{2}
\leq K \sum_{m=0}^{\infty} r^{m} (\sqrt{2})^{m} \Vert f_{m} \Vert_{1}
\leq K \sum_{m=0}^{\infty} (r\sqrt{2})^{m} \Vert f \Vert_{1} < \infty \, ,
\end{multline*}
where in the first inequality we used that $y \in \ell_{2}$ and \eqref{eq:H2}, in the second one we used Lemma~\ref{bonami} and in the last one that the projection
is a contraction. \\
Take now $z \in \ell_{2} \cap  \mathbb{D}^{\infty}$; then $\big( \sum_ {n=n_{0}}^{\infty} \vert z_{n} \vert^{2} \big)^{\frac{1}{2}} < r$ for some $n_{0}$. We define
$x=(0, \ldots , 0 , z_{n_{0}}, z_{n_{0}+1}, \ldots )$; clearly $x \in r B_{\ell_{2}} \cap \mathbb{D}^{\infty}$ and
$x \in \mon H_{1}(\mathbb{T}^{\infty})$, then Lemma~\ref{annalen} implies $z \in \mon H_{1} (\mathbb{T}^{\infty})$.\\
For the upper inclusion, by \cite[9.1~Theorem]{CoGa86} $H_{p}^{1}(\mathbb{T}^{\infty})
= H_{2}^{1}(\mathbb{T}^{\infty})$ with equivalent norms for each $1 \leq p < \infty$. This, together with Theorem~\ref{final}, gives
\[
\mon H_{p} (\mathbb{T}^{\infty}) \subseteq \mon H_{p}^{1}(\mathbb{T}^{\infty})  = \mon H_{2}^{1}(\mathbb{T}^{\infty})
\subseteq \ell_{2} \cap \mathbb{D}^{\infty} \, . \qedhere
\]
\end{proof}

\begin{remark} \label{evaluacion}
Let $\mathcal{P}_{\text{fin}}$ be the space of functions given by $\sum_{\alpha\in J} c_{\alpha} z^{\alpha}$ for $z \in \mathbb{C}^{\mathbb{N}}$, where
$J$ is some finite set of multi-indices. The evaluation mapping $\delta_{z} : \mathcal{P}_{\text{fin}} \rightarrow \mathbb{C}$ given by $\delta_{z}(f)=f(z)$ is
clearly well defined for each $z \in \mathbb{D}^{\infty}$. The space $\mathcal{P}_{\text{fin}}$ can be identified with the subspace of $H_{p}(\mathbb{T}^{\infty})$ of 
trigonometrical polynomials,
and one of the main problems considered in \cite{CoGa86} is for which $z$'s can $\delta_{z}$ be extended continuously to the whole $H_{p}(\mathbb{T}^{\infty})$.
This can be reformulated as to describe the following set
\[
 \{ z \in \mathbb{D}^{\infty} \colon \exists c_{z} \forall f \in \mathcal{P}_{\text{fin}} \, , \, \, \vert f(z) \vert \leq c_{z} \Vert f \Vert_{p} \} \, .
\]
Since for each $f  \in \mathcal{P}_{\text{fin}}$ and every $\alpha$ we have $\hat{f}(\alpha)=c_{\alpha}$, the previous set can be written as
\begin{equation} \label{Ep}
  \Big\{ z \in \mathbb{D}^{\infty} \colon \exists c_{z} \forall f \in \mathcal{P}_{\text{fin}} \, , \, \,
\Big\vert \sum_{\alpha} \hat{f}(\alpha) z^{\alpha} \Big\vert \leq c_{z} \Vert f \Vert_{p} \Big\} \, .
\end{equation}
In \cite[8.1 Theorem]{CoGa86} it is shown that for $1 \leq p < \infty$ the set in \eqref{Ep} is exactly $\ell_{2} \cap \mathbb{D}^{\infty}$.\\
By a closed-graph argument, for each $1 \leq p < \infty$ a sequence $z$ belongs to $\mon H_{p}(\mathbb{T}^{\infty})$ if and only if there exists $c_{z}$ such that for every
$f \in H_{p}(\mathbb{T}^{\infty})$
\begin{equation} \label{Baire}
\sum_{\alpha \in \mathbb{N}_{0}^{(\mathbb{N})}} \vert \hat{f}(\alpha) z^{\alpha} \vert
\leq c_{z} \Big( \int_{\mathbb{T}^{\infty}} \big\vert f( w) \big\vert^{p} dm(w)\Big)^{\frac{1}{p}} \,.
\end{equation}
This implies
\begin{equation} \label{mondom rewritten}
 \mon H_{p} (\mathbb{T}^{\infty}) =   \Big\{ z \in \mathbb{D}^{\infty} \colon \exists c_{z} \forall f \in \mathcal{P}_{\text{fin}} \, , \, \,
\sum_{\alpha} \vert  \hat{f}(\alpha) z^{\alpha} \vert \leq c_{z} \Vert f \Vert_{p} \Big\} \, .
\end{equation}
In view of  \eqref{mondom rewritten} we have that $\mon H_{p} (\mathbb{T}^{\infty})$ is contained in the set in  \eqref{Ep}. Then the upper inclusion
in Theorem~\ref{general}-\textit{(\ref{general p})} follows from \cite[8.1 Theorem]{CoGa86}. The proof we presented here is independent from that in \cite{CoGa86}.
But the lower inclusion in Theorem~\ref{general}-\textit{(\ref{general p})} is stronger than the result in \cite{CoGa86}.\\
The problem of describing the set in \eqref{Ep} for $p=\infty$ remains open in\cite{CoGa86}. We can give now a partial answer:
our Theorem~\ref{Leonhard} gives (using the notation from Section~\ref{sec:comments}),
\begin{align*}
 (\ell_{2,0}\cap \mathbb{D}^\infty) \cup \frac{1}{\sqrt{2}} B_{\ell_{2,\infty}} & \subseteq \frac{1}{\sqrt{2}} \mathbf{B} \\
\subseteq &
 \Big\{ z \in \mathbb{D}^{\infty} \colon \exists c_{z} \forall f \in \mathcal{P}_{\text{fin}} \, , \, \,
\Big\vert \sum_{\alpha} \hat{f}(\alpha) z^{\alpha} \Big\vert \leq c_{z} \Vert f \Vert_{\infty} \Big\} \, .
\end{align*}
\end{remark}

\begin{remark}
 The isomorphic Bohr abscissa $\varrho ( \mathcal{E})$ of a space of Dirichlet series $\mathcal{E}$ is defined (see \cite{BaCaQu06}) as the infimum
of the $\sigma \geq 0$ such that $\sum_{n} \vert a_{n} \vert n^{-\sigma} < \infty$ for all Dirichlet series in $\mathcal{E}$. Furthermore, it is said that the
Bohr abscissa is attained if $\sum_{n} \vert a_{n} \vert n^{-\varrho ( \mathcal{E})} < \infty$ for every  Dirichlet series in $\mathcal{E}$. Then
\cite[Theorem 1.1]{BaCaQu06} shows that $\varrho ( \mathcal{H}_{p} )=1/2$ for every $1 \leq p \leq \infty$ but only for $p=\infty$ is attained;
i.e. $\varrho ( \mathcal{H}_{p} )$ is not attained  for every $1 \leq p < \infty$ ($\mathcal{H}_{p}$ is the space of Dirichlet series whose associated power
series expansions belong to $H_{p}(\mathbb{T}^{\mathbb{N}})$). \\
The fact that it is not attained for $1 \leq p < \infty$ has a straightforward interpretation in terms of convergence of monomial expansions. Suppose the Bohr abscissa
is attained for some $p$. This would mean that $\sum_{n} \vert a_{n} \vert n^{-1/2} < \infty$ for every  Dirichlet series in $\mathcal{H}_{p}$. Then
$\sum_{\alpha} \vert a_{p^{\alpha}} \vert (p^{\alpha})^{-1/2} = \sum_{\alpha} \vert a_{p^{\alpha}} \vert \frac{1}{(p^{1/2})^{\alpha}} < \infty$
for every  Dirichlet series in $\mathcal{H}_{p}$. By doing $a_{p^{\alpha}} = \hat{f}(\alpha)$ this is equivalent to
$\sum_{\alpha} \vert \hat{f}(\alpha) \vert \frac{1}{(p^{1/2})^{\alpha}} < \infty$ for every $f \in H_{p}(\mathbb{T}^{\mathbb{N}})$ or, in other words,
$\big( \frac{1}{p_{n}^{1/2}} \big)_{n} \in \mon H_{p}(\mathbb{T}^{\mathbb{N}}) = \ell_{2} \cap \mathbb{D}^{\mathbb{N}}$. But this is not true.
This gives an alternative proof of the already known fact \cite[Theorem~1.1]{BaCaQu06} that the abscissa is not attained.
\end{remark}

\section{Representation of Hardy spaces}

We have seen in Proposition~\ref{H intftys} how, like in the finitely dimensional case, the Hardy space $H_{\infty} (\mathbb{T}^{\infty})$ can be represented as a
space of holomorphic functions on $\mathbb{D}_{0}^{\infty}$. In \cite[10.1 Theorem]{CoGa86} it is proved that every element of $H_p(\mathbb{T}^\infty)$ can be represented by an holomorphic function of bounded type on $\mathbb{D}^{\infty} \cap \ell_{2}$. A characterization of the holomorphic functions coming from elements of $H_p(\mathbb{T}^\infty)$  can be given for $1\leq p<\infty$, in terms of the following Banach space
\begin{multline*}
H_{p}(\mathbb{D}^{\infty})   = \big\{ g \in H(\mathbb{D}^{\infty} \cap \ell_{2}) \, \colon \, \\
\Vert g \Vert_{H_{p}(\mathbb{D}^{\infty})} =
\sup_{n \in \mathbb{N}} \sup_{0<r<1} \Big( \int_{\mathbb{T}^{n}} \vert g(rw_{1}, \ldots , r w_{n}) \vert^{p} dm_{n}(w_{1}, \ldots , w_{n}) \Big)^{\frac{1}{p}} < \infty \big\} \, .
\end{multline*}
Then we have

\begin{theorem} \label{Hps}
For each $1\leq  p<\infty$ the mapping
$\phi : H_{p}(\mathbb{T}^{\infty}) \rightarrow   H_{p}(\mathbb{D}^{\infty})$ defined by
$\phi(f)(z)= \sum_{\alpha \in \mathbb{N}_{0}^{(\mathbb{N})}} \hat{f}(\alpha) z^{\alpha}$ for $z \in \mathbb{D}^{\infty} \cap \ell_{2}$
is an isometry, that is onto if $1< p<\infty$, and   $\phi(H_{1}(\mathbb{T}^{\infty}))=\overline{\spa}^{\Vert \cdot \Vert_{1}}\{z^{\alpha} :\alpha \in \mathbb{N}_{0}^{(\mathbb{N})}\}$.

\end{theorem}
\begin{proof}. Fix $1\leq p<\infty$.
Let us begin by noting that for each fixed $n \in \mathbb{N}$ the mapping $\phi_{n} : H_{p}(\mathbb{T}^{n}) \rightarrow H_{p}(\mathbb{D}^{n})$
given by $\phi_{n}(f)(z)= \sum_{\alpha \in \mathbb{N}_{0}^{n}} \hat{f}(\alpha) z^{\alpha}$ is an isometric isomorphism, where $H_{p}(\mathbb{D}^{n})$ denotes
the Banach space of all holomorphic functions $g : \mathbb{D}^{n} \rightarrow \mathbb{C}$ such that
\[
  \Vert g \Vert_{H_{p}(\mathbb{D}^{n})}=
\sup_{0<r<1} \Big( \int_{\mathbb{T}^{n}} \vert g(rw_{1}, \ldots , rw_{n}) \vert^{p} dm_{n}(w_{1}, \ldots , w_{n}) \Big)^{\frac{1}{p}} < \infty \, .
\]
We show in first place that $\phi$ is well defined and a contraction.
Fix $f \in H_{p}(\mathbb{T}^{\infty})$; we know from Theorem~\ref{general} that $\sum_{\alpha \in  \in  \mathbb{N}_{0}^{(\mathbb{N})}} \vert \hat{f}(\alpha) z^{\alpha} \vert < \infty$ for every
$z \in \mathbb{D}^{\infty} \cap \ell_{2}$, hence the series defines a G\^{a}teaux-differentiable function on $\mathbb{D}^{\infty} \cap \ell_{2}$. We denote the $m$-th Taylor
polynomial of $\phi(f)$ at $0$ by $P_{m}$. Since
\[
  P_{m}(z)= \sum_{\substack{\alpha \in \mathbb{N}_{0}^{n} \\ \vert \alpha \vert =m}}  \hat{f}(\alpha) z^{\alpha}
\]
for $z \in \ell_{2}$, we deduce from \eqref{normas final} that $P \in \mathcal{P}(^{m} \ell_{2})$ and hence $\phi (f)$ defines a holomorphic function on
$\mathbb{D}^{\infty} \cap \ell_{2}$ (see e.g. \cite[Example~3.8]{Di99}). Let us see now that it actually belongs to $H_{p}(\mathbb{D}^{\infty})$.
Following the notation in Proposition~\ref{H intftys} we define for each $n$
\[
  f_{n}(w_{1}, \ldots , w_{n})  = \int_{\mathbb{T}^{\infty}} f(w_{1}, \ldots , w_{n}, \tilde{w}_{n}) dm(\tilde{w}_{n}) \, ,
\]
where $(w_{1}, \ldots , w_{n})\in \mathbb{T}^n$. It is well defined by the Fubini theorem, since $L_{p}(\mathbb{T}^{\infty}) \subseteq L_{1}(\mathbb{T}^{\infty})$. On the other hand,
using the H\"{o}lder inequality and Fubini theorem we get
\begin{align*}
  \int_{\mathbb{T}^{n}} & \vert f_{n} (w_{1}, \ldots , w_{n}) \vert ^{p} dm_{n}  (w_{1}, \ldots , w_{n})  \\
& = \int_{\mathbb{T}^{n}} \Big\vert \int_{\mathbb{T}^{\infty}} f(w_{1}, \ldots , w_{n}, \tilde{w}_{n}) dm(\tilde{w}_{n}) \Big\vert^{p} dm_{n}  (w_{1}, \ldots , w_{n})  \\
& \leq \int_{\mathbb{T}^{n}} \Big( \int_{\mathbb{T}^{\infty}} \vert f(w_{1}, \ldots , w_{n}, \tilde{w}_{n}) \vert dm(\tilde{w}_{n}) \Big)^{p} dm_{n}  (w_{1}, \ldots , w_{n})  \\
& \leq \int_{\mathbb{T}^{n}} \Big( \int_{\mathbb{T}^{\infty}} \vert f(w_{1}, \ldots , w_{n}, \tilde{w}_{n}) \vert^{p} dm(\tilde{w}_{n}) \Big) dm_{n}  (w_{1}, \ldots , w_{n})
\end{align*}
and this implies $f \in L_{p} (\mathbb{T}^{n})$ and $\Vert f_{n} \Vert_{p} \leq \Vert f \Vert_{p}$ for all $n$. Moreover, for $\alpha \in \mathbb{Z}^{n}$ we have
(again using Fubini) $\hat{f}_{n}(\alpha) = \hat{f} (\alpha)$ and $f_{n} \in H_{p} (\mathbb{T}^{n})$. Then $\Vert \phi( f_{n}) \Vert_{H_{p}(\mathbb{D}^{n})}
= \Vert  f_{n} \Vert_{p} \leq \Vert f \Vert_{p}$ for all $n$ and we have
\begin{equation} \label{IV}
\sup_{n \in \mathbb{N}} \sup_{0<r<1} \int_{\mathbb{T}^{n}} \big\vert \sum_{\alpha \in  \mathbb{N}_{0}^{n}}  \hat{f}(\alpha) (rw)^{\alpha} \big\vert  
dm_{n}  (w_{1}, \ldots , w_{n}) \leq \Vert f \Vert_{p} < \infty \, .
\end{equation}
Clearly $\phi(f)(z_{1}, \ldots , z_{n},0 \ldots) = \sum_{\alpha} \hat{f}(\alpha)
z^{\alpha}$ for every $(z_{1}, \ldots , z_{n}) \in  \mathbb{D}^{n}$ and by \eqref{IV} this implies $\phi(f) \in H_{p}(\mathbb{D}^{\infty})$ and
\[
  \Vert \phi(f) \Vert_{H_{p}(\mathbb{D}^{\infty})} \leq \Vert f \Vert_{p} \, .
\]
Now we are going to show that it is also an isometry onto if $1<p<\infty$. Let $g \in H_{p}(\mathbb{D}^{\infty})$ and $g_{n}$ its restriction to the
first $n$ variables. Then, by definition $g_{n} \in H_{p}(\mathbb{D}^{n})$ and $\Vert g_{n} \Vert_{H_{p}(\mathbb{D}^{n})} \leq
\Vert g \Vert_{H_{p}(\mathbb{D}^{\infty})}$. Let us take $f_{n} = \phi_{n}^{-1}(g_{n}) \in H_{p} (\mathbb{T}^{n})$ and define
$\tilde{f}_{n} : \mathbb{T}^{\infty} \to \mathbb{C}$ by $\tilde{f}_{n}(w) = f_{n}(w_{1}, \ldots , w_{n})$ for $w \in \mathbb{T}^{\infty}$. Since we can do this
for every $n$, we have a sequence $(\tilde{f}_{n})_{n}$ contained in the ball of $H_{p} (\mathbb{T}^{\infty})$ centered in $0$ and with radius
$\Vert g \Vert_{H_{p}(\mathbb{D}^{\infty})}$, that is a weak$^{*}$-compact set. Since $L_{q}(\mathbb{T}^{\infty})$ is separable the weak$^{*}$-topology
is metrizable and then there exists a subsequence $(\tilde{f}_{n_{k}})_{k}$ that weak$^{*}$ converges to some $f \in H_{p} (\mathbb{T}^{\infty})$. For each
$\alpha \in \mathbb{N}_{0}^{(\mathbb{N})}$ we then have
\[
  \hat{f}(\alpha) = \langle f , w^{\alpha} \rangle = \lim_{k} \langle \tilde{f}_{n_{k}} , w^{\alpha} \rangle = \hat{\tilde{f}}_{n_{k}}(\alpha) = c_{\alpha} (g) \,.
\]
Hence $\phi(f)= g$ and, moreover, $\Vert f \Vert_{p} \leq \Vert g \Vert_{H_{p}(\mathbb{D}^{\infty})}$, which completes the proof for $1<p<\infty$.\\

For the case $p=1$ we observe that $\phi(w^{\alpha})=z^{\alpha}$  for every $\alpha \in \mathbb{N}_{0}^{(\mathbb{N})}$; then the mapping
\[
\phi:\spa\{[ w \in \mathbb{T}^{\infty} \mapsto w^{\alpha} ] :\alpha \in \mathbb{N}_{0}^{(\mathbb{N})}\} \longrightarrow
\spa\{[ z \in \mathbb{D}^{\infty} \mapsto z^{\alpha} ] :\alpha \in \mathbb{N}_{0}^{(\mathbb{N})}\}
\]
is a linear isomorphism between normed spaces. In fact, since each trigonometrical polynomial $P$ depends on a finite number of variables, we
have $\Vert \phi(P) \Vert_{1}=\Vert P \Vert_{1}$; thus $\phi$ extends to an isometry (that coincides with the original $\phi$)
\[
\phi:\overline{\spa}^{\Vert \cdot \Vert_{1}} \{w^{\alpha} :\alpha \in \mathbb{N}_{0}^{(\mathbb{N})}\}\longrightarrow
\overline{\spa}^{\Vert \cdot \Vert_{1}} \{z^{\alpha} :\alpha \in \mathbb{N}_{0}^{(\mathbb{N})}\}\, .
\]
Finally it is a well known fact that $H_{1}(\mathbb{T}^\infty)=\overline{\spa}^{\Vert \cdot \Vert_{1}}\{w^{\alpha} : \alpha \in \mathbb{N}_{0}^{(\mathbb{N})}\}$
(this follows for example from F\'ejer's Theorem for several variables and the fact that continuous functions  depending only in finitely many variables
are dense in   $L_{1}(\mathbb{T}^\infty)$). This completes the proof.
\end{proof}

\noindent We do not know whether $\overline{\spa}^{\Vert \cdot \Vert_{1}}\{z^{\alpha} :\alpha \in \mathbb{N}_{0}^{(\mathbb{N})}\}$ coincides
with $H_{1}(\mathbb{D}^{\infty})$ or not. But note that any polynomial in $H_{1}(\mathbb{D}^{\infty})$ is actually contained in
$\overline{\spa}^{\Vert \cdot \Vert_{1}}\{z^{\alpha} :\alpha \in \mathbb{N}_{0}^{(\mathbb{N})}\}$ since $H_{1}^m(\mathbb{T}^\infty)=H_2^m(\mathbb{T}^\infty)$
for every $m \in \mathbb{N}$ and $\phi( H_2(\mathbb{T}^\infty))=H_2(\mathbb{D}^\infty)$.

\noindent defant@mathematik.uni-oldenburg.de \\ frerick@uni-trier.de \\ maestre@uv.es \\ psevilla@mat.upv.es


\begin{thebibliography}{10}

\bibitem{BaCaQu06}
R.~Balasubramanian, B.~Calado, and H.~Queff\'{e}lec.
\newblock The {B}ohr inequality for ordinary {D}irichlet series.
\newblock {\em Studia Math.}, 175(3):285--304, 2006.

\bibitem{Ba02}
F.~Bayart.
\newblock Hardy spaces of {D}irichlet series and their composition operators.
\newblock {\em Monatsh. Math.}, 136(3):203--236, 2002.

\bibitem{BoHi31}
H.~F. Bohnenblust and E.~Hille.
\newblock On the absolute convergence of {D}irichlet series.
\newblock {\em Ann. of Math. (2)}, 32(3):600--622, 1931.

\bibitem{Bo13_Goett}
H.~Bohr.
\newblock \"{U}ber die {B}edeutung der {P}otenzreihen unendlich vieler
  {V}ariabeln in der {T}heorie der \textit{{D}irichlet}--schen {R}eihen
  $\sum\,\frac{a_n}{n^s}$.
\newblock {\em Nachr. Ges. Wiss. G\"{o}ttingen, Math. Phys. Kl.}, pages
  441--488, 1913.

\bibitem{Bo13}
H.~Bohr.
\newblock \"{U}ber die gleichm\"{a}{\ss}ige {K}onvergenz {D}irichletscher
  {R}eihen.
\newblock {\em J. Reine Angew. Math.}, 143:203--211, 1913.

\bibitem{CoGa86}
B.~J. Cole and T.~W. Gamelin.
\newblock Representing measures and {H}ardy spaces for the infinite polydisk
  algebra.
\newblock {\em Proc. London Math. Soc. (3)}, 53(1):112--142, 1986.

\bibitem{DaGa89}
A.~M. Davie and T.~W. Gamelin.
\newblock A theorem on polynomial-star approximation.
\newblock {\em Proc. Amer. Math. Soc.}, 106(2):351--356, 1989.

\bibitem{Br08}
R.~de~la Bret\`{e}che.
\newblock Sur l'ordre de grandeur des polyn\^omes de {D}irichlet.
\newblock {\em Acta Arith.}, 134(2):141--148, 2008.

\bibitem{DeFr11}
A.~Defant and L.~Frerick.
\newblock The {B}ohr radius of the unit ball of $\ell_{p}^{n}$.
\newblock {\em J. Reine Angew. Math.}, to appear, 2011.

\bibitem{DeFrOrOuSe11}
A.~Defant, L.~Frerick, J.~Ortega-Cerd\`a, M.~Ouna\"{\i}es, and K.~Seip.
\newblock The {B}ohnenblust--{H}ille inequality for homogeneous polynomials is
  hypercontractive.
\newblock {\em Annals of Mathematics}, 174(1):485--497, 2011.

\bibitem{DeGaMa10}
A.~Defant, D.~Garc{\'{\i}}a, and M.~Maestre.
\newblock New strips of convergence for {D}irichlet series.
\newblock {\em Publ. Mat.}, 54(2):369--388, 2010.

\bibitem{DeGaMaPG08}
A.~Defant, D.~Garc{\'i}a, M.~Maestre, and D.~P\'{e}rez-Garc{\'i}a.
\newblock Bohr's strip for vector valued {D}irichlet series.
\newblock {\em Math. Ann.}, 342(3):533--555, 2008.

\bibitem{DeMaPr09}
A.~Defant, M.~Maestre, and C.~Prengel.
\newblock Domains of convergence for monomial expansions of holomorphic
  functions in infinitely many variables.
\newblock {\em J. Reine Angew. Math.}, 634:13--49, 2009.

\bibitem{DeScSe}
A.~Defant, U.~Schwarting, and P.~Sevilla-Peris.
\newblock Queff\'elec numbers.
\newblock {\em in preparation}.

\bibitem{Di99}
S.~Dineen.
\newblock {\em Complex analysis on infinite-dimensional spaces}.
\newblock Springer Monographs in Mathematics. Springer-Verlag London Ltd.,
  London, 1999.

\bibitem{Fl99}
K.~Floret.
\newblock Natural norms on symmetric tensor products of normed spaces.
\newblock {\em Note Mat.}, 17:153--188 (1999), 1997.

\bibitem{Fo87}
J.~J.~F. Fournier.
\newblock Mixed norms and rearrangements: {S}obolev's inequality and
  {L}ittlewood's inequality.
\newblock {\em Ann. Mat. Pura Appl. (4)}, 148:51--76, 1987.

\bibitem{Ha75}
L.~A. Harris.
\newblock Bounds on the derivatives of holomorphic functions of vectors.
\newblock In {\em Analyse fonctionnelle et applications (Comptes Rendus Colloq.
  Analyse, Inst. Mat., Univ. Federal Rio de Janeiro, Rio de Janeiro, 1972)},
  pages 145--163. Actualit\'{e}s Aci. Indust., No. 1367. Hermann, Paris, 1975.

\bibitem{HeLiSe97}
H.~Hedenmalm, P.~Lindqvist, and K.~Seip.
\newblock A {H}ilbert space of {D}irichlet series and systems of dilated
  functions in {$L^2(0,1)$}.
\newblock {\em Duke Math. J.}, 86(1):1--37, 1997.

\bibitem{Hilbert_Gesam_3}
D.~Hibert.
\newblock {\em Gesammelte Abhandlungen (Band 3)}.
\newblock Verlag von Julius Springer, Berlin, 1935.

\bibitem{Hi09}
D.~Hilbert.
\newblock {W}esen und {Z}iele einer {A}nalysis der unendlichvielen
  unabh\"{a}ngigen {V}ariablen.
\newblock {\em Rend. del Circolo mat. di Palermo}, 27:59--74, 1909.

\bibitem{Ka85}
J.-P. Kahane.
\newblock {\em Some random series of functions}, volume~5 of {\em Cambridge
  Studies in Advanced Mathematics}.
\newblock Cambridge University Press, Cambridge, second edition, 1985.

\bibitem{KoQu01}
S.~V. Konyagin and H.~Queff\'{e}lec.
\newblock The translation {$\frac12$} in the theory of {D}irichlet series.
\newblock {\em Real Anal. Exchange}, 27(1):155--175, 2001/02.

\bibitem{MaQu10}
B.~Maurizi and H.~Queff\'{e}lec.
\newblock Some remarks on the algebra of bounded {D}irichlet series.
\newblock {\em J. Fourier Anal Appl}, DOI 10.1007/s00041-009-9112-y, 2010.

\bibitem{Qu95}
H.~Queff\'{e}lec.
\newblock H. {B}ohr's vision of ordinary {D}irichlet series; old and new
  results.
\newblock {\em J. Anal.}, 3:43--60, 1995.

\bibitem{Ru69}
W.~Rudin.
\newblock {\em Function theory in polydiscs}.
\newblock W. A. Benjamin, Inc., New York-Amsterdam, 1969.

\bibitem{To13}
O.~Toeplitz.
\newblock \"{U}ber eine bei den {D}irichletschen {R}eihen auftretende {A}ufgabe
  aus der {T}heorie der {P}otenzreihen von unendlichvielen
  {V}er\"{a}nderlichen.
\newblock {\em Nachrichten von der K\"{o}niglichen Gesellschaft der
  Wissenschaften zu G\"{o}ttingen}, pages 417--432, 1913.

\bibitem{Wo91}
P.~Wojtaszczyk.
\newblock {\em Banach spaces for analysts}, volume~25 of {\em Cambridge Studies
  in Advanced Mathematics}.
\newblock Cambridge University Press, Cambridge, 1991.

\end{thebibliography}
\end{document}